\documentclass[12pt]{amsart}

\font\cyr=wncyr10

\usepackage{epsfig}
\usepackage{mathtools}
\usepackage{amsmath}
\usepackage{amssymb}
\usepackage{amsfonts}

\usepackage{fancyheadings}
\usepackage{amscd}
\usepackage{graphicx}
\usepackage{color}
\usepackage{verbatim}
\usepackage{bbm}
\usepackage{upgreek}
\RequirePackage{textcmds}\relax
\ProvideTextCommandDefault{\cprime}{\tprime}
\newtheorem{theorem}{Theorem}[section]

\newtheorem{proposition}[theorem]{Proposition}
\newtheorem{corollary}[theorem]{Corollary}
\newtheorem{question}[theorem]{Question}

\theoremstyle{definition}
\newtheorem{definition}[theorem]{Definition}

\theoremstyle{statement}

\newtheorem{problem}[theorem]{Problem}

\theoremstyle{remark}
\newtheorem{remark}[theorem]{Remark}

\numberwithin{equation}{section}

\newcommand{\CT}{\mathcal{T}}

\newcommand{\CS}{\mathcal{S}}

\def \N {\mathbb N}
\def \O {\mathcal O}

\def \D {\mathcal D}
\def \Z {\mathbb Z}

\def \EA {E_{\mathsf A}}
\def \F {\mathcal F}
\def \H {\mathcal H}
\def \EF {\mathsf E\mathcal F}
\def \EH {\mathsf E\mathcal H}
\def \Q {\mathcal Q}

\def \P {\mathcal P}
\def \M {\mathcal M}
\def \MGX {\mathcal M (X, G)}

\def \MGY {\mathcal M (Y, G)}

\def \htop {h_{\mathsf{top}}}

\def \sq {sequence}
\def \xg {$(X,G)$}

\def \tl {topological}
\def \im {invariant measure}
\def \inv {invariant}
\def \ds {dynamical system}

\def \htop{h_{\mathsf{top}}}
\def \Hom {\mathsf{Hom}}
\def \id {\mathsf{Id}}

\def \usc {upper semicontinuous}
\def \se {superenvelope}
\def \ens {entropy structure}
\def \zd {zero-dimen\-sio\-nal}
\def \CT {\mathcal{T}}

\def \eps {\varepsilon}

\def \N {\mathbb N}

\def \Z {\mathbb Z}

\def \T {\mathbb T}

\def \F {\mathcal F}
\def \H {\mathcal H}
\def \EF {\mathsf E\mathcal F}
\def \EH {\mathsf E\mathcal H}
\def \Q {\mathcal Q}

\def \P {\mathcal P}
\def \M {\mathcal M}

\def \htop {h_{\mathsf{top}}}

\def \sq {sequence}
\def \xg {$(X,G)$}

\def \tl {topological}
\def \im {invariant measure}
\def \inv {invariant}
\def \ds {dynamical system}

\def \htop{h_{\mathsf{top}}}
\def \Hom {\mathsf{Hom}}
\def \id {\mathsf{Id}}

\def \usc {upper semicontinuous}
\def \se {superenvelope}
\def \ens {entropy structure}
\def \qt {quasitiling}
\def \CT {\mathcal{T}}
\def \CS {\mathcal{S}}

\makeindex

\begin{document}



\title{The symbolic extension theory in topological dynamics}

\author{Tomasz Downarowicz and Guohua Zhang}

\address{\vskip 2pt \hskip -12pt Tomasz Downarowicz}

\address{\hskip -12pt Faculty of Pure and Applied Mathematics, Wroclaw University of Science and Technology, Wybrze\.ze Wyspia\'nskiego 27, 50-370 Wroc\l aw, Poland}

\email{downar@pwr.edu.pl}

\address{\vskip 2pt \hskip -12pt Guohua Zhang}

\address{\hskip -12pt School of Mathematical Sciences and Shanghai Center for Mathematical Sciences, Fudan University, Shanghai 200433, China}

\email{chiaths.zhang@gmail.com}



\parindent=10pt
\begin{abstract}
In this survey we will present the symbolic extension theory in topological dynamics, which was built over the past twenty years.
\end{abstract}

\maketitle

\setcounter{tocdepth}{2}

\tableofcontents


\section{History of symbolic dynamics and symbolic extensions}\label{sI}

The classical definition of a dynamical system encompassing the dynamical
systems of Newtonian mechanics is that of a motion whose parameters
evolve as functions of the time in accordance with a
system of differential equations. One simplification in its study is to discretize time, so
that the state of the system is observed only at discrete ticks of a clock,
which leads to the study of the iterates of a single
transformation. Symbolic dynamics arose as an attempt to study systems by means
of discretizing space as well as time.
The origins of symbolic dynamics come from representing geodesics on surfaces of constant negative curvature by symbolic sequences. It
 can be traced to the work of Hadamard \cite{Hadamard} in 1898, where he constructed surfaces in $\mathbb{R}^3$ of negative curvature, and discovered that
geodesics on these surfaces can be described by sequences of symbols by recording the successive sides of a given fundamental region cut by the
geodesic. Geodesics can also be represented by
using continued fraction expansions of the end
points of the geodesic at infinity and the so-called Gauss
reduction theory, which may even go back to works of Gauss \cite{Gauss} and Dirichlet \cite{Dirichlet}. Since then this idea was developed by Morse, Artin, Koebe,
Nielsen and Hedlund in the 1920s and 30s, and became an important tool in the study of dynamical systems, of which geodesic flows on Riemannian manifolds of negative sectional
curvature represent a major class of examples. See \cite{AF, KU} and the references therein for more historical information.

In \cite{MH38} Morse
and Hedlund named the subject of symbolic dynamics.
The basic idea is to divide up the set
of possible states into a finite number of pieces, and keep track of which
piece the state of the system lies in at every tick of the clock. Each piece
is associated with a symbol and in this way the evolution of the system
is described by an infinite sequence of symbols. This leads to a symbolic
dynamical system corresponding to the original system. In the 1960s and 70s, many symbolic systems were constructed for general hyperbolic systems using Markov partitions developed by Sinai \cite{Sinai68}, Ratner \cite{R69}, Adler and Weiss \cite{AW70} and Bowen \cite{Bow70, Bow73}.
For an introduction to symbolic dynamics,
its basic ideas and techniques see the books \cite{LM95} by Lind and Marcus and \cite{K98} by Kitchens.

\medskip

In some sense, topology provides a static picture of an object, while a topological dynamical system corresponds to a motion picture or more precisely a high definition video of the evolution of the object. In this analogy, a symbolic system can be compared to a video with a crude, low definition. It turns out that a suitable ``low definition video'' is capable of providing the complete information about the evolution of the observed system.

\emph{One of the most important tasks in the theory of symbolic dynamics may be giving criteria for a lossless digitalization of a general abstract dynamical system.}

\medskip

  In the classical ergodic theory of $\Z$-actions (i.e.\ actions by iterates of a single transformation), Krieger's Generator Theorem \cite{Kr} resolves the problem completely using the Kolmogorov--Sinai entropy: any invertible measurable system (on a standard probability space) with finite Kolmogorov--Sinai entropy $h$  admits a finite generating partition consisting of $l=\lfloor2^h\rfloor+1$ atoms and hence is isomorphic to a subshift over $l$ symbols equipped with some shift-\inv\ measure.
The result was later generalized by \v Sujan \cite{Su} to free actions of countable discrete infinite amenable groups (with later improvements by Rosenthal \cite{R} and Danilenko and Park \cite{DP}), and even obtained recently for actions of general countable groups by Seward \cite{Sew}.

In topological dynamics of $\mathbb{Z}$-actions (i.e. actions by iterates of a single self-homeomorphism of a compact metric space), it is natural to ask an analogous question: is it possible, and if yes then how many symbols are needed, to losslessly encode a topological synamical system (with finite topological entropy) in a subshift? Obviously, in general, it is impossible to represent a system by a \tl ly conjugate subshift even if the system has finite topological entropy. A classical Hedlund's result \cite{He} states that a topological $\mathbb{Z}$-system is
topologically conjugate to a subshift if and only if it is expansive and the state space is
zero dimensional. However, a wide range of systems are \emph{semi-conjugate} to subshifts, in the sense that they are topological factors of symbolic systems, in other words, they admit symbolic extensions. Clearly, a symbolic extension carries the complete information about the underlying topological dynamical systems and hence can be considered its lossless digitalization. In this manner, we are led to studying symbolic extensions, finding criteria for their existence and minimizing their topological entropy.

 \medskip

A natural generalization of $\Z$-actions are actions (both measure-theoretic and topological) of countable groups. \emph{Throughout this paper we let $G$ be a discrete countably infinite group with the unity $e$.}

 By a \emph{\tl\ action} (or just an \emph{action}) of $G$ on a compact metric space $X$, we will mean a homomorphism from $G$ into $\Hom(X)$, where $\Hom(X)$ denotes the group of all homeomorphisms $\phi:X\to X$, that is an assignment $g\mapsto \phi_g$ such that $\phi_{gg'} = \phi_g\circ\phi_{g'}$ for every $g,g'\in G$.
It follows automatically that $\phi_{e}=\id$ (the identity homeomorphism) and that $\phi_{g^{-1}}=(\phi_g)^{-1}$ for every $g\in G$. Such an action will be denoted by \xg.\footnote{Although $G$ may act on $X$ in many different ways, we will usually fix just one such action, hence the notation $(X, G)$ should not lead to a confusion. Sometimes we also denote by $(X, T)$ an action of $\mathbb{Z}$ on a compact metric space $X$, where $\phi_1 = T: X\rightarrow X$ is a homeomorphism of the space $X$, and then the group $\mathbb{Z}$ in fact acts as the family
 $\{T^n: n\in \mathbb{Z}\}$ of homeomorphisms.} Another term used for $(X,G)$ is a \emph{\tl\ dynamical system} (or briefly a \emph{system}). From now on, to reduce the multitude of symbols used in this survey, we will write $g(x)$ in place of $\phi_g(x)$. The same applies to subsets $\mathsf A\subset X$: $g(\mathsf A)$ will replace $\phi_g(\mathsf A)$.
 A Borel set $\mathsf A\subset X$ is called \emph{\inv} if $g(\mathsf A)=\mathsf A$ for every $g\in G$.

An important example of a $G$-action is the \emph{shift action on finitely many symbols}. Let $\Lambda$ be a finite set\footnote{Usually we assume that $\Lambda$ contains more than one element, otherwise the associated symbolic system is trivial.} considered as a discrete topological space (in this context $\Lambda$ is called the \emph{alphabet}), and let
$\Lambda^G=\{x=(x_g)_{g\in G}:\ \forall_{g\in G}\ x_g\in\Lambda\}$
be equipped with any metric compatible with the product topology. Then $\Lambda^G$ is a compact metric space and $G$ acts on it naturally by shifts:
$$
\text{if \ }x=(x_f)_{f\in G}\text{ \ and \ }g\in G\text{ \ then \ }g(x)=(x_{fg})_{f\in G}.
$$
The system $(\Lambda^G, G)$ is called the \emph{full shift} (over $\Lambda$) while any nonempty closed invariant subset $Y\subset\Lambda^G$ (regarded with the shift action) is called a \emph{subshift} or \emph{symbolic system}.

If $(X,G)$ and $(Y,G)$ are actions of the same group on two (not necessarily different) spaces, and there exists a continuous surjection $\pi:Y\to X$ compatible with the action (i.e. such that $\pi\circ g(y)=g\circ\pi(y)$ for any $g\in G$ and $y\in Y$), then $(X,G)$ is called a \emph{\tl\ factor of $(Y,G)$} and $(Y,G)$ is called a \emph{\tl\ extension of $(X,G)$}, and the above continuous surjection $\pi$ will be referred to as a \emph{topological factor map} or \emph{topological extension}.
In what follows, we will skip the adjective ``\tl", and when the acting group is fixed and its action on given spaces is understood, we will also skip it in the denotation of the \ds s (i.e.\  we will use the letters $X$ and $Y$ in the meaning of $(X,G)$ and $(Y,G)$).
It has to be remarked that one system, say $X$, may be a factor of another, say $Y$, via many different factor maps. Since this may lead to a confusion, we will often use the phrase \emph{$X$ is a factor of $Y$ via the map $\pi$} and \emph{$Y$ is an extension of $X$ via the map $\pi$}. If $\pi$ is additionally injective, in which case it is a homeomorphism between $Y$ and $X$, then we call them \emph{\tl ly conjugate}.
 If $Y$ is a symbolic system and $X$ is a factor of $Y$ via the map $\pi$, then we say that $Y$ is \emph{a symbolic extension of $X$} and that $\pi$ is \emph{a symbolic extension of $X$} (or just \emph{a symbolic extension}).

\medskip

 Thus, from the point of view of \tl\ dynamics, the following question arises naturally.

 \begin{question} \label{202002082222}
 When does a system $X$ admit a symbolic extension?
 \end{question}

\begin{remark} \label{202002111720}
If the group $G$ is amenable, in order for a system $X$ to admit a symbolic extension, it must necessarily have finite topological entropy. For actions of amenable groups, it is a basic fact that the topological entropy of a system is less than or equal to that of its extension, and that any symbolic system has finite topological entropy.
\end{remark}

 The first result concerning symbolic extensions in topological dynamics is
due to Reddy \cite{Reddy68}. It says that every expansive
$\mathbb{Z}$-action\footnote{Recall that a $\mathbb{Z}$-action $(X, T)$ is \emph{expansive} if there exists $\delta> 0$ such that once $x_1$ and $x_2$ are different points in $X$ then $d (T^n x_1, T^n x_2)> \delta$ for some $n\in \mathbb{Z}$.} admits a symbolic extension, which is proved via the language of topological generators.
   Note that expansiveness and topological generators can be introduced naturally for actions of any discrete countable group $G$, and the same proof in \cite{Reddy68} shows that every expansive $G$-action admits a symbolic extension.
Obviously expansiveness is a very strong requirement in topological dynamics, which is not satisfied even by some ``simple" systems (such as a circle rotation). Moreover, while the property of admitting a symbolic extension is obviously inherited by factors, expansiveness is not.

In the 1970s all known examples of $\mathbb{Z}$-actions with finite topological entropy seemed to admit symbolic extensions.
One of the spectacular applications of symbolic extensions occurs in the studies
of hyperbolic systems. Using Markov partitions, such systems can be lifted to
subshifts of finite type, which allows to apply symbolic dynamical methods to
study hyperbolic systems. For more information see Bowen's classic book \cite{Bo77}. However, very little was known for general systems even when we restrict to $\mathbb{Z}$-actions.

The question whether all $\mathbb{Z}$-actions with finite topological entropy indeed admit symbolic
extensions has been presumably puzzling many people between the years 1970 and 1990.
 Around 1989, Joe Auslander addressed this question to Mike Boyle,
and then Boyle came up
with the negative answer, by constructing an appropriate zero-dimensional example \cite{Boy91}. This example was presented at the Adler conference in 1991, but remained unpublished until the year 2002, when it was included in the survey \cite{BFF}.

Since the seminal works by Moulin Ollagnier \cite{MO85} and Ornstein and Weiss \cite{OW} (after Kolmogorov--Sinai measure-theoretic entropy \cite{Ko59, Sinai59} for measurable $\mathbb{Z}$-actions and Adler-Konheim-McAndrew topological entropy \cite{AKM65} for topological $\mathbb{Z}$-actions), the entropy theory of amenable group actions was developed by many researchers in \cite{Da01, DooZ12, DooZ15, GTW, HYZ, KL07, KL09, MP79, MP82, RW, ST80, WZ}, see also the survey \cite{Weiss03} by Benjy Weiss.

\emph{From now on additionally we assume $G$ to be amenable} (see \S \ref{amenable} for the definition of amenability).
Inspired by results in classical ergodic theory of $\Z$-actions, one may ask:

\begin{question} \label{202002081406}
Assume that the action $(X, G)$ admits symbolic extensions. What is the infimum of topological entropies of all the symbolic
extensions of \xg?
\end{question}

By analogy to the measure-theoretic case, a naive guess would be that this infimum simply equals the topological entropy of the action $(X, G)$. In fact this does happen for $\mathbb{Z}$-actions with zero topological entropy. Boyle proved in \cite{Boy91} that any zero entropy $\mathbb{Z}$-action admits a symbolic extension with zero topological entropy. However, another example provided by Boyle (a simpified version of the above mentioned example of \zd\ $\mathbb{Z}$-actions which admit no symbolic extensions) shows that, even if a $\mathbb{Z}$-action does admit a symbolic extension,
there may exist an essential gap between the topological entropy of the system and that of any of its symbolic extensions. This gap was called by Boyle \emph{residual entropy}.

In order to formalize Question \ref{202002081406}, one introduces the following definition.
Here we denote by $\mathbf{h_\text{top}} (X, G)$ (or just $\mathbf{h_\text{top}} (X)$ if no confusion occurs) the topological entropy of an action $(X, G)$.

\begin{definition} \label{202002092121}
The \emph{topological symbolic extension entropy} of an action $(X, G)$ is defined as
$$\mathbf{h_\text{sex}} (X, G)\footnote{It is also denoted by $\mathbf{h_\text{sex}} (X, T)$ in the setting of a $\mathbb{Z}$-action $(X, T)$.}=
\begin{cases}
\inf\ \{\mathbf{h_\text{top}} (Y): Y\ \text{is a symbolic extension of}\ X\},\\
\infty,\ \ \ \ \ \ \ \ \ \ \text{if $X$ admits no symbolic extensions}.
\end{cases}$$
\end{definition}

As shown in \cite{BFF}, it does not seem feasible to answer completely Questions \ref{202002082222} and \ref{202002081406} and compute the topological symbolic extension entropy of an action, only by means of topological and symbolic methods, though the invariant $\mathbf{h_\text{sex}}$ is of purely topological-dynamical nature. In fact, the problem has been solved using methods of ergodic theory, and up to date there seems to be no way around it.

Since we have assumed $G$ to be amenable, any topological action $(X, G)$ admits invariant Borel probability measures (or briefly invariant measures)\footnote{Recall that for the action $(X, G)$ we say a Borel probability measure $\mu$ supported on $X$ \emph{invariant} if $g \mu= \mu$ for each $g\in G$, where $(g \mu) (A)\doteq \mu (g^{- 1} A)$ for all Borel subsets $A\subset X$.} supported on $X$. Denote by $\mathcal{M} (X, G)$\footnote{In the case of a $\mathbb{Z}$-action $(X, T)$, we also denote it by $\mathcal{M} (X, T)$.} the set of all invariant measures supported by $(X, G)$, which is equipped with the weak-star topology.
Then $\mathcal{M} (X, G)$ is always a nonempty compact metric convex space, whose set of extreme points coincides with the set of ergodic measures.\footnote{Recall that $\mu\in \mathcal{M} (X, G)$ is \emph{ergodic} if $\mu (A)$ equals either $0$ or $1$ whenever a Borel set $A\subset X$ is invariant.}

Ergodic theory and topological dynamics have exhibited a remarkable parallelism. The global variational principle concerning entropy (for $\mathbb{Z}$-actions see \cite{Good71, Good69, M75} and for amenable group actions see \cite{MP82, ST80}) provides a useful bridge between entropy theories of topological dynamical systems and measure-theoretic ones. Thus it does not seem very surprising that, in order to compute the topological symbolic extension entropy of a topological action \xg, one needs to focus not only on the topological entropy of the system and its symbolic extensions, but also on the Kolmogorov--Sinai entropies of all invariant measures supported by these systems.

\medskip

This is indeed the case, as shown in \cite{D01, D0} and finally in \cite{BD}, where the Questions~\ref{202002082222} and~\ref{202002081406} and the problem of computing the invariant $\mathbf{h_\text{sex}}$ are completely solved in the generality of $\mathbb{Z}$-actions. The solution is contained in the so-called \emph{Symbolic Extension Entropy Theorem} which provides a criterion for a topological $\Z$-action to admit a symbolic extension and allows to predict the Kolmogorov--Sinai entropies of all invariant measures supported by any possible symbolic extension (for details see \cite[Theorem 5.5]{BD} or \cite[Theorem 9.2.1]{D01}, for further developments see also \cite{BD06, BuD}). In particular, combined with a result from \cite{Buzzi97}, one obtains the striking result that every $C^\infty$ map on a compact $C^\infty$ manifold has a symbolic extension with the same \tl\ entropy (and additional good properties). A natural next step is to consider $C^r$ maps with $1\le r< \infty$. The general theory for this situation was systematically explored in \cite{DN}: roughly speaking a typical $C^1$ system (in some class) on a Riemannian manifold of dimension $d\ge 2$ admits no symbolic extensions
at all. This problem was further discussed in \cite{Bu102, Bu101, DF}. We remark that \cite{DN}  leaves the problem open whether there exists symbolic extensions for all $C^r$ maps on Riemannian manifolds with $1< r< \infty$. The authors formulate a conjecture which postulates the positive answer with a possible upper bound for the topological symbolic extension entropy. The conjecture has been proved in the case of interval or the circle maps in \cite{DM}, and for a $C^2$ surface diffeomorphism in \cite{Bu}, and then for a $C^r$ diffeomorphism with $r > 1$ on a compact three-dimensional manifold
in \cite{BL}. In all of these works a crucial role is played by the notion of \emph{entropy structure}, developed and discussed systematically in \cite{D0}.
All the above developments lead to the creation of a relatively new \emph{theory of symbolic extensions} (of $\mathbb{Z}$-actions) in topological dynamics. For its history see also the survey \cite{D2010}, and for a detailed explanation of the theory and proofs of most of the abstract results (excluding smooth dynamics) see the book \cite{D1}.

As most of the entropy theory, including not only topological and Kolmogorov--Sinai entropy but also Shannon--McMillan--Breiman Theorem, Krieger's Generator  Theorem, and so on, can be very well extended from $\mathbb{Z}$-actions to actions of discrete countable amenable groups, it seems reasonable to believe that the realm of countable amenable group actions is the most natural environment to carry over the theory of symbolic extensions as well.

At the beginning of the year 2011 we (the authors of this survey) were  very optimistic and hoped to be able to generalize the theory of symbolic extensions to actions of general amenable groups without serious difficulties. In fact, the first step in this direction has already been made for $\mathbb{Z}^k$-actions by Gutman \cite{Gu1} in 2011.  But it soon turned out that, except for some steps of the generalization which are indeed not very difficult, the case of general amenable group actions is much more complicated than expected. The difficulties arise already at the level of zero entropy systems and depend on some unsolved  problems associated to intrinsic properties of the groups in question. These problems are discussed by the authors and their collaborator Dawid Huczek, in the papers \cite{DH, DHZ, H} written over the past ten years, where the background is laid for the final solution. The main question is finally resolved recently in \cite{DZ}. However, in order to obtain results fully analogous to those known for $\Z$-actions we have to make compromises. In \cite{DZ}, for actions of general countable amenable groups we prove a slightly ``deficient'' version of the Symbolic Extension Entropy Theorem, in which symbolic extensions are replaced by the so-called quasi-symbolic extensions, defined as extensions in form of topological joinings of subshifts with some zero-dimensional zero entropy tiling systems. These zero-dimensional zero entropy tiling systems are completely determined by the group $G$, and in general it remains unknown whether such a system admits a symbolic extension (even though it has entropy zero). This is the reason why we have to join it in an unchanged form with the symbolic extension. If the amenable group $G$ has an additional property called the \emph{comparison property}, or is residually finite, then in \cite{DZ} we are able to prove the uncompromised Symbolic Extension Entropy Theorem fully analogous to the one known for $\Z$-actions.

\medskip

This manuscript consists of two parts, where
we introduce the symbolic extension theory for $\mathbb{Z}$-actions and for general amenable group actions, respectively.

Part 1 is structured as follows: In \S \ref{general-stru} we introduce some basic functional analytic properties which will be used to characterize symbolic extension entropy. In \S \ref{stru} we discuss the \ens\ introduced in detail in \cite{D0}. Roughly speaking, entropy structure reflects the emergence of Kolmogorov--Sinai entropies of \im s of the considered systems at refining scales.
In \S \ref{four} we state the main results of the symbolic extension theory for $\mathbb{Z}$-actions. This includes
the Symbolic Extension Entropy Theorem (which utilizes \ens\ to characterize the symbolic extension entropy function), the Symbolic Extension Entropy Variational Principle (interpreting topological symbolic extension entropy via symbolic extension entropy function), and the characterization of asymptotically $h$-expansive $\mathbb{Z}$-actions by means of symbolic extensions.

The structure of Part 2 is as follows: In \S \ref{amenable} we make some necessary preparations for discussing countable amenable group actions. In \S \ref{prepa} we introduce the quasitiling and tiling systems of a given abstract amenable group, along with the so-called tiled entropy of \zd\ actions. These tools are crucial in our Quasi-Symbolic Extension Entropy Theorem. In \S \ref{theory} we present our main results about symbolic extensions for amenable group actions, mainly consisting of Theorem \ref{quasi}, the mentioned Quasi-Symbolic Extension Entropy Theorem. In \S \ref{im-theory} we state the uncompromised Symbolic Extension Entropy Theorem for two special classes of amenable groups, that is, amenable groups which are residually
finite or which have the comparison property. Then in the final section \S \ref{ques} we list some open questions for further study of the theory.

 \part{The theory of symbolic extensions for $\mathbb{Z}$-actions} \label{1part}

\emph{Throughout this part we will focus only on $\mathbb{Z}$-actions}.

In this part we will focus on abstract results valid for general $\Z$-actions (by homeomorphisms) on compact metric spaces.  We remark that even though there are plenty of interesting and valuable results concerning symbolic extensions $C^r$ maps (with $1\le r< \infty$) on smooth Riemannian manifolds, we do not include them in this survey.

 \medskip

Though this part relies heavily on the (Adler--Konheim--McAndrew) topological entropy and Kolmogorov--Sinai entropy for $\mathbb{Z}$-actions,
we assume that the reader is familiar with these notions and we skip their definitions. The interested reader can find them in any textbook about entropy of dynamical systems, see for example recent books \cite{D1, Gla}.

\section{General theory of a structure and its superenvelope} \label{general-stru}

Let $\M$ be a compact metric space. We say that the function $f: \mathcal{M}\rightarrow \mathbb{R}$ is \emph{upper semicontinuous} (or briefly \emph{u.s.c.}) if $\limsup_{\mu'\rightarrow \mu} f (\mu')\le f (\mu)$ for every $\mu\in \M$, equivalently, if $\{\mu\in \M: f (\mu)\ge r\}$ for every $r\in \mathbb{R}$ is a closed subset; and that $f$ is \emph{lower semicontinuous} (or briefly \emph{l.s.c.}) if $- f$ is u.s.c.

The following fact (cited from \cite{BD, D1}) is very useful in later analysis.

\begin{proposition} \label{exchanging}
Let $(f_k)_{k\in \N}$ be a decreasing sequence of u.s.c.\ functions on $\mathcal{M}$. Then
$$\inf_{k\in \N} \sup_{x\in \mathcal{M}} f_k (x)= \sup_{x\in \mathcal{M}} \inf_{k\in \N} f_k (x).$$
\end{proposition}

\begin{definition} By a \emph{structure} on $\M$ we will mean any nondecreasing \sq\ of commonly bounded nonnegative functions on $\M$, $\F =(f_k)_{k\ge 0}$ with $f_0\equiv 0$.
\end{definition}

Clearly, the pointwise \emph{limit function} $f=\lim_k f_k$ exists and is nonnegative and bounded.
\begin{definition}
Two structures $\F =(f_k)_{k\ge 0}$ and $\F' = (f'_k)_{k\ge 0}$ are said to be \emph{uniformly equivalent} if
$$
\forall_{\varepsilon>0,\,k_0\ge 0}\ \exists_{k\ge 0}\ \ (f'_k>f_{k_0}-\varepsilon \text{ and } f_k>f'_{k_0}-\varepsilon).
$$
\end{definition}
It is immediate that uniform equivalence is an equivalence relation between structures. It is also obvious that uniformly equivalent structures have a common limit function.

Now let $f: \mathcal{M}\rightarrow \mathbb{R}_{\ge 0}$ be a bounded function on $\M$. By the \emph{\usc\ envelope} (or briefly \emph{u.s.c.\ envelope}) of $f$ we shall mean the function $\tilde f$ defined on $\M$
$$
\tilde f(\mu) = \limsup_{\mu'\to\mu}f(\mu') = \inf_{U\ni\mu}\sup\{f(\mu'):\mu'\in U\} = \inf\{g \text{ continuous and } g\ge f\},
$$
where $\mu,\mu'\in\M$ and $U$ ranges over all open neighborhoods of $\mu$ (all of the above formulas are equivalent). Note that $\tilde f\ge f$. We also define the \emph{defect of $f$} as the difference $\overset{...}f = \tilde f-f$. Clearly $\tilde f$ is always u.s.c., and then $f$ is u.s.c.\ if and only if $f=\tilde f$ or equivalently $\overset{...}f \equiv 0$.

\begin{definition}
 Let $\F=(f_k)_{k\ge 0}$ be a structure defined on a compact metric space $\M$.
 \begin{enumerate}

 \item A nonnegative function $E$ defined on $\M$ is called a \emph{superenvelope} of $\F$, if $E\ge f_k$ for each $k\ge 0$ and the defects $\overset{..............}{E-f_k}$ tend pointwise to zero.

\item We say that $\F$ has \emph{\usc\ differences} (or briefly \emph{u.s.c.\ differences}), if the difference functions $f_{k+1}-f_k$ are all u.s.c.\ for every $k\ge 0$.

\item Let $\M$ be (in addition to being compact metric) convex with a convex metric. We say that a structure $\F$ is \emph{affine} if all $f_k$ are affine functions on $\M$.
 \end{enumerate}
  \end{definition}

 \begin{remark}
A priori a structure may have no \se s as defined above, and so by default, the constant infinity function is added to the collection of  \se s of any structure.
 \end{remark}

Let $E$ be a finite superenvelope (i.e. not the constant infinity function) of a structure $\F=(f_k)_{k\ge 0}$ defined on a compact metric space $\M$ and let $f$ be the limit function of $\F$. Then $E- f$ equals the nonincreasing limit function of $(\widetilde{E-f_k})_{k\ge 0}$, in particular, it is u.s.c., hence bounded. Thus $E$ is also bounded.

\medskip

Here are some basic facts about superenvelopes of a structure (for details see for example \cite[Lemmas 8.1.10 and 8.1.12, Theorems 8.1.25 (2) and 8.2.5]{D1}):

\begin{proposition}\label{fact2}
The following statements hold:
\begin{enumerate}

  \item The infimum of all \se s of a structure $\F$ is still a \se\ of $\F$ (in the
  extreme case it is the constant infinity function). This \emph{minimal \se} of
  $\F$ will be denoted by $\EF$.

	\item Uniformly equivalent structures have the same collection of \se s (and hence have the same minimal \se).

	\item If a structure $\F=(f_k)_{k\ge 0}$ has u.s.c.\ differences then a (finite) function $E$ is its \se\ if and only if $E-f_k$ is nonnegative and u.s.c.\ for every $k\ge 0$ (in particular $E = E-f_0$ is then u.s.c.).

	\item If $\F$ is an affine structure with u.s.c.\ differences, defined on a metrizable Choquet simplex\footnote{By a
\emph{metrizable Choquet simplex} we mean a nonempty compact
convex set which possesses a convex metric, and every its element has
a unique representation as the integral average of its extreme points.
We remark that, given any $G$-action $(X, G)$ with $G$ additionally amenable, the space $\mathcal{M} (X, G)$, endowed with the weak-star topology, is a metrizable Choquet simplex.}, then $\EF$ coincides with the pointwise infimum of all affine \se s of $\F$ (in particular the function $\EF$ is concave).
\end{enumerate}
\end{proposition}

We will also need the following terminology.
 Let $\pi:Y\to X$ be a continuous surjection between compact metric spaces. If $f$ is a bounded nonnegative function on $X$, then we define its \emph{lift of $f$ (against $\pi$)} as the composition $f\circ\pi$, which is a function defined on $Y$, and we will denote it by the same letter $f$ (if it causes no confusion). Obviously the lifted function is \emph{constant on fibers}, that is, it is constant on the sets $\pi^{-1}(x)$ for every $x\in X$. Now if $g$ is a bounded nonnegative function on $Y$, then we define the \emph{push-down of $g$ (along $\pi$)} as the function $g^\pi$ on $X$ given by $g^{\pi}(x) = \sup\{g(y): y\in\pi^{-1}(x)\}$ for every $x\in X$. Thus lifting reverses the operation of pushing down exclusively for functions constant on fibers.

The following facts are either immediate or proved in \cite[Facts A.1.26 and A.2.22]{D1}:

\begin{proposition}\label{pushdown}
The following statements hold:
\begin{enumerate}

\item The operation of lifting preserves u.s.c., l.s.c. (and hence continuity) of a function. If $Y$ and $X$ are both compact metric convex spaces, and $\pi$ is affine, then lifting preserves concavity, convexity (and hence affinity) of a function. The same holds for pushing down functions which are constant on fibers.

\item The pushing down preserves u.s.c.\ of a function. If $Y$ and $X$ are
both compact metric convex spaces, and $\pi$ is affine, then pushing down preserves concavity of a function. If, moreover, $Y$ and $X$ are both Choquet simplices and $\pi$ sends extreme points of $Y$ to extreme
points of $X$, then pushing down also preserves affinity of a function.
\end{enumerate}
\end{proposition}

\section{Entropy structure of $\mathbb{Z}$-actions} \label{stru}

One of key ingredients of the Symbolic Extension Entropy Theorem is the entropy structure, which reflects the emergence of different Kolmogorov--Sinai entropies of invariant measures of the considered systems at refining scales.
By Remark \ref{202002111720}, the system has finite topological entropy if it admits a symbolic extension, and so in this section we are only interested in $\mathbb{Z}$-actions with finite topological entropy. Entropy structure of $\mathbb{Z}$-actions has been discussed systematically in \cite{D0}, see also \cite{D1} for more details about it.

 \medskip

Entropy structure is introduced in two steps. At first we assume that the space $X$ is zero-dimensional. In this case we consider a refining sequence of finite clopen partitions (this dependence will be later eliminated). A partition is \emph{clopen} if all its atoms are clopen (i.e.\ both closed and open). A \sq\ $(\mathcal{P}_k)_{k\in \mathbb{N}}$ of partitions is \emph{refining} if the maximum diameter of all elements of $\mathcal{P}_k$ goes to zero with $k\rightarrow \infty$ and furthermore each partition $\mathcal{P}_{k+ 1}$ \emph{refines} $\mathcal{P}_k$ (i.e.\  each element of $\mathcal{P}_{k+ 1}$ is contained in some element of $\mathcal{P}_k$). Clearly, in any compact zero-dimensional space $X$ a refining \sq\ of finite clopen partitions exists.

\begin{definition}\label{enszero}
Let $(X, T)$ be a zero-dimenaional system with finite topological entropy and let $(\mathcal{P}_k)_{k\in \mathbb{N}}$ be a refining \sq\ of finite clopen partitions of $X$. Then an \emph{entropy structure} of $(X, T)$ is the sequence $\H=(h_k)_{\ge 0}$ of functions defined on $\mathcal{M} (X, T)$, by $h_k (\mu)= h_\mu (T, \mathcal{P}_k)$\footnote{We denote by $h_\mu (T, \mathcal{P}_k)$ the Kolmogorov--Sinai $\mu$-entropy of the partition $\mathcal{P}_k$ with respect to the invariant measure $\mu\in \mathcal{M} (X, T)$.} for every $k\in \mathbb{N}$ and all $\mu\in \mathcal{M} (X, T)$.
\end{definition}

\begin{proposition} \label{202002121502}
If $(X, T)$ is a zero-dimensional system with finite topological entropy, then:
\begin{enumerate}

\item regardless of the choice of the refining \sq\ of finite clopen partitions $(\P_k)_{k\in\N}$, the associated entropy structure of $(X,T)$ is an affine structure with u.s.c.\ differences.

\item the entropy structures obtained for different refining \sq s of finite clopen partitions are uniformly equivalent.
\end{enumerate}
\end{proposition}

 In general, we cannot require the existence of a refining \sq\ of finite clopen partitions if $X$ is not zero dimensional, and so we need finite Borel partitions with \emph{small boundaries}, that is, such that the boundary of each atom of the partition has measure zero for every invariant measure. As an intermediate case between zero-dimensional and general, suppose that $(X, T)$ is a system with finite topological entropy and admits a refining sequence $(\mathcal{P}_k)_{k\in \mathbb{N}}$ of finite Borel partitions of $X$, with small boundaries. Then an \emph{entropy structure} of $(X, T)$ is the sequence $\H=(h_k)_{k\ge 0}$ with $h_k (\mu)= h_\mu (T, \mathcal{P}_k)$ for every $k\in \mathbb{N}$ and all $\mu\in \mathcal{M} (X, T)$.

The existence of a refining sequence of finite Borel partitions with small boundaries is essential in the above definition. Besides when $X$ is zero-dimensional, such a \sq\ exists if either $X$ is finite-dimensional and additionally the set of all periodic points of $(X, T)$ is zero-dimensional (see \cite{BFF}, the result is essentially contained in \cite{Kul}), or $(X, T)$ has finite topological entropy and admits a minimal\footnote{A system $(X, T)$ is \emph{minimal} if for every $x\in X$ the subset $\{T^n x: n\in \mathbb{Z}_+\}$ is dense in $X$.} factor consisting of infinitely many points
(see \cite{L99, LW00}).

Nevertheless, it is still possible that in some system with finite topological entropy there are no refining sequences of finite Borel partitions with small boundaries. The general case is handled with the help of the following concept introduced by Ledrappier in \cite{led}.

\begin{definition} \label{prin}
Let $\pi: Y\rightarrow X$ be an extension, where $X$ has finite topological entropy.
We say that the extension $\pi$ (also $Y$) is \emph{principal}\footnote{Principal extension was defined in \cite{led} originally via the language of relative measure-theoretic entropy, here we provide an equivalent definition in the setting of systems with finite topological entropy.} if $h_\nu (Y)= h_{\pi \nu} (X)$\footnote{We will write $h_\nu (Y, S)$ or $h_\nu (Y)$ the Kolmogorov--Sinai entropy of $(Y, S)$ with respect to an invariant measure $\nu\in \mathcal{M} (Y, S)$.
Note that $\pi$ induces naturally a continuous affine surjection, denoted still by $\pi$ which will not cause any confusion, from $\mathcal{M} (Y,
S)$ to $\mathcal{M} (X, T)$, with $(\pi \nu) (B)= \nu (\pi^{- 1} B)$ for every $\nu\in \mathcal{M} (Y, S)$ and all Borel subsets $B\subset Y$.
The map $\pi$ in fact has the additional property that it sends extreme points to extreme points.} for every $\nu\in \mathcal{M} (Y, S)$.
\end{definition}

The following basic fact provides us an auxilliary zero-dimensional system.

\begin{proposition} \label{principal}
Any system with finite topological entropy admits a zero-dimensional principal extension.
\end{proposition}

Entropy structure of a general dynamical system is given as follows:

\begin{definition}\label{ensf}
The \emph{\ens} of a system $(X, T)$ with finite topological entropy is any structure $\H=(h_k)_{k\ge 0}$ on $\mathcal{M} (X, T)$ such that for any \zd\ principal extension $\pi': X'\to X$ and any entropy structure $\H'$ on $\mathcal{M} (X', T')$, the structure $\H$ lifted against $\pi'$ (given by $(h_k\circ \pi')_{k\ge 0}$) is uniformly equivalent to $\H'$.
\end{definition}

The existence of such a structure is nontrivial but true, see for example \cite{D0, D1}. Once the existence of at least one such structure $\H$ is guaranteed, it becomes obvious that if $X$ is zero-dimensional then the \ens\ given by Definition \ref{enszero} is consistent with Definition~\ref{ensf}, and that any other structure defined on $\mathcal{M} (X, T)$ is an entropy structure if and only if it is uniformly equivalent to $\H$. In order to obtain a notion which depends exclusively on the system $(X,T)$ (and does not depend on the choice of, say, a principal zero-dimensional extension or a refining \sq\ of finite clopen partitions in that extension), we will replace the ``individual'' \ens\ by the entire uniform equivalence class, and call this class the \emph{entropy structure}. Nevertheless, instead of saying that a particular structure $\H$ \emph{belongs to} the \ens\ we will keep saying that $\H$ \emph{is} an \ens.

It is clear that as \ens, undertood as a uniform equivalence class, is an invariant of \tl\ conjugacy in the sense that the entropy structure of one system carried over via a conjugating map is an entropy structure of the other system.

We remark that \cite{DS2} provides a characterization of structures defined on metrizable Choquet simplices which can be realized as an entropy structure of some \tl\ \ds s.

Superenvelopes of the \ens\ $\H$ of a system $(X,T)$ are used to characterize the entropy functions of symbolic extensions of the system. Note that, by Proposition \ref{fact2}, the collection of all superenvelopes does not depend on the choice of a particular \ens\ from a uniform equivalence class, in other words, superenvelopes of an \ens\ remain well defined  when \ens\ is understood as an equivalence class.
On the other hand, for technical convenience, we are free to select from the uniform equivalence class which constitutes the \ens\ a particular representative
which enjoys additional good properties. For example, we may choose (if possible) an affine \ens\ with u.s.c.\ differences. The following theorem makes such a choice possible:

\begin{theorem} \label{202002121830}
Any system with finite topological entropy admits an affine \ens\ with u.s.c.\ differences.
\end{theorem}

\begin{remark} \label{natural}
All results of Part \ref{1part} work for surjective $\mathbb{N}$-actions, i.e.\ actions by iterates of a continuous surjective self-map $T$ of a compact metric space $X$. Transferring the $\mathbb{N}$-action $(X,T)$ to a $\Z$-action is achieved by means of a \emph{natural extension} $(\widetilde{X}, \widetilde{T})$ which is a $\mathbb{Z}$-action defined as
$$\widetilde{X}= \{(x_n)_{n\in \Z}: T x_n= x_{n+ 1}, \forall n\in \Z\},$$ where the map $\widetilde{T}$ acts on the space $\widetilde{X}$ by the shift:
$\widetilde{T}((x_n)_{n\in \Z})= (x_{n+1})_{n\in \Z}$. Note that the notion of a principal extension applies without modification to $\mathbb{N}$-actions, and that $(\widetilde{X}, \widetilde{T})$ (viewed as an $\N$-action) is a principal extension of  $(X,T)$. Nonetheless, one has to bear in mind that in such a generalization, by a symbolic extension of $(X,T)$ we will understand a symbolic extension of $(\widetilde{X}, \widetilde{T})$ viewed as a $\mathbb{Z}$-action. In particular, all symbolic extensions are by definition two-sided subshifts ($\Z$-actions).
\end{remark}

\section{The symbolic extension theory of $\mathbb{Z}$-actions} \label{four}

 In this section we introduce the main results
of the theory of symbolic extensions for $\mathbb{Z}$-actions, including the Symbolic Extension Entropy
Theorem, the Symbolic Extension Entropy Variational Principle, and the characterization of asymptotic h-expansiveness via symbolic extensions.

\medskip

Let us firstly consider a zero-dimensional system $(X, T)$. Every such system can be given the following (conjugate) \emph{array representation}.
We choose a \sq\ of finite clopen partitions $(\P_k)_{k\in\N}$ such that $(\P_{[1,k]})_{k\in \N}$ forms a refining sequence, where $\P_{[1,k]}$ is the join partition generated by $\P_1, \dots, \P_k$.
For each $k\in \mathbb{N}$, let $\Lambda_k$ be a set of labels bijectively associated to the atoms of $\P_k$, so that $\P_k = \{P_a: a\in\Lambda_k\}$. Now according to the action $T$ we introduce:
\begin{align*}
&\forall_{k\in\N}\,\forall_{n\in \Z}\ (x_{k,n}:=a\in\Lambda_k \iff T^n (x)\in P_a\in\P_k),\\
&\forall_{k\in\N}\hskip 26pt \pi_k(x):=(x_{k,n})_{\,n\in \Z}\in\Lambda_k^\Z,\\
&\hskip 54pt \pi(x):=(\pi_k(x))_{k\in\N}=(x_{k,n})_{k\in\N,n\in \Z}\in\prod_{k\in\N}\Lambda_k^\Z.
\end{align*}
We let $X_k$ denote the image of $X$ by the map $\pi_k$ which is a subshift over $\Lambda_k$, and call it \emph{the $k$th layer} of $X$. Denote by $X_{[1,k]}$ the projection of $X$ onto the first $k$ layers, which is a subshift over the product alphabet $\Lambda_{[1,k]}=\prod_{i=1}^k\Lambda_i$. Because all partitions $\P_k$ are clopen, the maps $\pi_k$ and $\pi$ are continuous. In fact, $X_{[1,k]}$ is a factor of $X$ via a factor map $\pi_{[1,k]}$ induced naturally by $\pi_1, \dots, \pi_k$, and $\pi$ is injective. In this manner $X$ is \tl ly conjugate to $\pi (X)$, where the action on $\pi (X)$ is by the simultaneous shift of all layers. We will call $\pi (X)$ an \emph{array representation} of $X$ and because we treat conjugate systems as one, we can imagine any \zd\ system in its array representation.

The following result from \cite{D01} provides an answer to Question \ref{202002081406} for a \zd\ $\mathbb{Z}$-action. Note the the functions $h_k$ below do not represent an \ens\ but the \sq\ of differences in an \ens; the paper \cite{D01} was written before the notion of \ens\ was coined.

\begin{theorem} \label{zd-sex}
With an array representation as above for a zero-dimensional system $(X, T)$, let $(h_k)_{k\in \mathbb{N}}$ be a sequence of nonnegative functions on $\mathcal{M} (X, T)$ as follows: $\forall \mu\in \mathcal{M} (X, T)$,
$$h_1 (\mu)= h_{\pi_1 \mu} (X_1)\ \ \ \text{and}\ \ \ h_k (\mu)= h_{\pi_{[1, k]} \mu} (X_{[1, k]})- h_{\pi_{[1, k- 1]} \mu} (X_{[1, k- 1]}), \forall k\ge 2.$$
Then
$$\mathbf{h_\text{sex}} (X, T)= \inf \left\|\sum_{k\in \mathbb{N}} f_k\right\|,$$
where $\|\cdot\|$ denotes the supremum norm of a function and the infimum runs over all sequences $(f_k)_{k\in \N}$ of continuous functions on $\mathcal{M} (X, T)$ such that $f_k\ge h_k$ for all $k\in \N$.
\end{theorem}

Consequently, for a \zd\ system $(X, T)$, one has $\mathbf{h_\text{sex}} (X, T)= \mathbf{h_\text{top}} (X, T)$ if, in particular,  $\mathbf{h_\text{top}} (X, T)=0$ (then $(X,T)$ has symbolic extensions with arbitrarily small entropy) or if $(X,T)$ has finite topological entropy and has only finitely many ergodic measures.  As mentioned before, we have a better result for zero-entropy systems: any zero entropy $\mathbb{Z}$-action (not only zero-dimensional) admits a symbolic extension with zero (not only arbitrarily small) topological entropy \cite{Boy91}. The above result does not imply, even when $(X, T)$ is \zd\ with $\mathbf{h_\text{sex}} (X, T)= \mathbf{h_\text{top}} (X, T)$, that there exists a symbolic extension $(Y, S)$ with $\mathbf{h_\text{top}} (Y, S)= \mathbf{h_\text{top}} (X, T)$. As we will see later, the existence of such an extension may indeed fail.

Thus, it is natural to ask when does the infimum in the definition of $\mathbf{h_\text{sex}}$ is achieved (i.e. becomes a minimum):

\begin{question} \label{202002131000}
Under what conditions does the action $(X, T)$ admit a symbolic extension $(Y, S)$ with
$\mathbf{h_\text{top}} (Y, S)= \mathbf{h_\text{sex}} (X, T)$?
\end{question}

A partial answer to this question will follow from the further discussion.
\medskip

Misiurewicz introduced in \cite{M73}  an important class of \tl\ $\Z$-systems called asymptotically $h$-expansive systems, and then characterized them in \cite{M} via the language of so-called topological conditional entropy. It was observed in \cite{D01} that, for a \zd\ system $(X, T)$, it is asymptotically $h$-expansive if and only if it admits an array representation with $\sum_{k\in \N} \mathbf{h_\text{top}} (X_k)< \infty$. Based on this, it was proved that each \zd\ asymptotically $h$-expansive $\mathbb{Z}$-system admits a symbolic extension with equal topological entropy. Note that it was remarked in \cite{D01} that this result was firstly proved by Boyle without the assumption of dimension zero, but the proof was never published until 2002 in \cite{BFF}.

\medskip

Theorem \ref{zd-sex} is not completely satisfactory because it seems rather difficult to compute $\mathbf{h_\text{sex}}$ by controlling the continuous functions $f_k$. Moreover, even though it describes some cases when $\mathbf{h_\text{sex}} (X, T)=\mathbf{h_\text{top}} (X,T)$, it is not clear that in such cases a principal symbolic extension exists. An alternative solution is to replace the functions $f_k$ by another, more manageable,
\sq\ of functions defined on the set of all invariant measures of the action. Entropy structure is such a \sq\ and indeed it can be used in this context giving much stricter results. This is the point, where it is necessary to start observing the entropy functions rather than topological entropies.

In order to compute $\mathbf{h_\text{sex}}$ more effectively, it is natural to present a refined version of Definition \ref{202002092121} involving invariant measures supported on the action as follows.

 \begin{definition} \label{202002092127}
 One can define the following functions on $\mathcal{M} (X, T)$:
 \begin{enumerate}

 \item If $(X, T)$ is a factor of $(Y, S)$ via the map $\pi$, then the \emph{extension entropy function} with respect to $\pi$, denoted by $h^\pi$, is given by
 $$\mu\mapsto \sup \{h_\nu (Y, S): \nu\in \mathcal{M} (Y, S)\ \text{and}\ \pi \nu= \mu\}$$ (in other words, $h^\pi$ is the push-down of the entropy functions on $\M(Y,S)$ to $\M(X,T)$).

 \item The \emph{symbolic extension entropy function} of $(X, T)$, denoted by $h_\text{sex}$, is given by
     $$\mu\mapsto \inf \{h^\pi (\mu): \text{$Y$ is a subshift and $\pi:Y\to X$ is a symbolic extension}\}.\footnote{It may happen that $(X, T)$ admits no symbolic extensions, thus here we use the convention $\inf \emptyset= + \infty$.}$$
 \end{enumerate}
 \end{definition}

One may ask the following question, similar to Question \ref{202002131000}:

 \begin{question} \label{202002140038}
Under what conditions does the action $(X, T)$ admit a symbolic extension $\pi$ with $h_\text{sex}\equiv h^\pi$?
\end{question}

The \emph{Symbolic Extension Entropy Theorem}, the main result obtained in \cite{BD}, which characterizes completely the extension entropy functions in symbolic extensions in terms of functional analytic properties of an entropy structure of the action, is as follows:

\begin{theorem} \label{seet}
Assume that $(X, T)$ is a system with finite topological entropy. Let $\mathcal{H}$ be an \ens\ of $(X, T)$ and let $E$ be a nonnegative finite function defined on $\mathcal{M} (X, T)$. Then:
\begin{enumerate}

\item \label{former} The function $E$ equals the extension entropy function $h^\pi$ with respect to some symbolic extension $\pi:Y\to X$ if and only if $E$ is a finite affine superenvelope of the structure $\mathcal{H}$.

\item \label{later} The function $h_\text{sex}$ is exactly the minimal superenvelope $\EH$, in particular, it is u.s.c.\ and concave.
\end{enumerate}
\end{theorem}

\begin{remark} \label{202002131530}
As the system $(X, T)$ always admits an affine entropy structure with u.s.c.\ differences (see Theorem \ref{202002121830}), \eqref{later} follows directly from \eqref{former} via Proposition \ref{fact2}.
\end{remark}

 Inspired by the global variational principle concerning entropy of $\mathbb{Z}$-actions, it is natural to expect that there is a similar variational principle concerning symbolic extension entropy in the setting of $\mathbb{Z}$-actions.
 In fact this is the case, see \cite{BD, D1}. \emph{Symbolic Extension Entropy Variational Principle} provides a satisfactory answer to Question \ref{202002081406} for $\mathbb{Z}$-actions.

 \begin{theorem} \label{sex-vp}
$\mathbf{h_\text{sex}} (X, T)= \sup \{h_\text{sex} (\mu): \mu\in \mathcal{M} (X, T)\}=\sup \{\EH (\mu): \mu\in \mathcal{M} (X, T)\}$.
 \end{theorem}

\begin{remark} \label{202002131630}
The variational principle can be proved by combining Proposition \ref{exchanging} with Theorem \ref{seet}, though it is not a direct consequence by simply exchanging suprema and infima.
\end{remark}

\begin{remark} With the help of Theorems \ref{seet} and \ref{sex-vp}, it can be proved that the topological symbolic extension entropy $\mathbf{h_\text{sex}}$ and the symbolic extension entropy function $h_\text{sex}$ both behave with respect to powers and products ``as one might expect". See \cite{BD06} for details.
\end{remark}

Observe that by Proposition \ref{pushdown} the extension entropy function is affine. Then by Theorems ~\ref{seet} and \ref{sex-vp} one can answer Questions \ref{202002131000} and \ref{202002140038} as follows.

\begin{theorem} \label{attain}
Let $(X, T)$ be a system with finite topological entropy and an \ens\ $\mathcal{H}$. Assume that it admits a symbolic extension. Then:
\begin{enumerate}
\item There exists a symbolic extension $\pi: (Y, S)\rightarrow (X, T)$ such that $\forall \mu\in \mathcal{M} (X, T)$ \ $h^\pi (\mu)= h_\text{sex} (\mu)$ if and only if the minimal superenvelope $\EH$ is an affine function on $\mathcal{M} (X, T)$.\footnote{By \cite[Theorem 4.6]{BD} this happens, in particular, if the set of all ergodic measures is closed.}
\item There exists a symbolic extension $\pi: (Y, S)\rightarrow (X, T)$ such that $\mathbf{h_\text{top}} (Y, S)= \mathbf{h_\text{sex}} (X, T)$ if and only if
    there exists an affine superenvelope $E_A$ of the structure $\mathcal{H}$ such that
    $$\sup_{\mu\in \mathcal{M} (X, T)} E_A (\mu)\footnote{Given such an affine superenvelope, there exists a symbolic extension $\pi: (Y, S)\rightarrow (X, T)$ such that $h^\pi= E_A$, and then one has that the supremum of the function $E_A$ is exactly the topological entropy of $(Y, S)$, by the classical variational principle concerning entropy of $\mathbb{Z}$-actions \cite{Good71, Good69, M75}.}= \sup_{\mu\in \mathcal{M} (X, T)} \EH (\mu).$$
\end{enumerate}
\end{theorem}

As another corollary of Theorem \ref{seet}, one has the following beautiful characterization of asymptotically $h$-expansive $\mathbb{Z}$-actions
 (see \cite{HYZ10, HYZ14} for more characterizations about this notion). We remark that the equivalence of \eqref{fir} $\Longleftrightarrow$ \eqref{third} was established firstly in \cite{BFF}.

\begin{theorem} \label{ahe}
Let $(X, T)$ be a system with finite topological entropy. Then the following statements are equivalent:
 \begin{enumerate}

 \item \label{fir} The action $(X, T)$ is asymptotically $h$-expansive.

 \item The action $(X, T)$ admits an \ens\ which converges uniformly to the entropy function $\mu\mapsto h_\mu (X, T)$ on $\mathcal{M} (X, T)$.

 \item \label{third} The action $(X, T)$ admits a principal symbolic extension.

 \item For every $\eps> 0$, the action $(X, T)$ admits a symbolic extension $(Y, S)$ with relative entropy at most $\eps$, that is, $h_\nu (Y| X)\le \eps, \forall \nu\in \mathcal{M} (Y, S)$.\footnote{The conditional entropy $h_\nu (Y| X)$ will be introduced in the setting of amenable group actions in \S \ref{amenable}.}

 \item The entropy function on $\mathcal{M} (X, T)$ is the minimal superenvelope of the entropy structure of $(X, T)$.
 \end{enumerate}
\end{theorem}

\part{The theory of symbolic extensions for amenable group actions}

Recall that \emph{we have assumed $G$ to be a discrete countably infinite amenable group with the unity $e$}.

\medskip

This part is devoted to
describing a possible theory of symbolic extensions for amenable group actions, which is mainly based on results in \cite{DH, DHZ, DZ, H}.
As we shall see, the case of general amenable group actions is much more complicated
than expected, for example, it even remains open till now if there exists a symbolic \emph{free}\footnote{The action $(X, G)$ is called \emph{free} provided that $g(x)=x$ for at least one $x\in X$ implies $g=e$. We remark that any extension of a free action is still a free action, also that the standard irrational rotation on the circle is a free $\mathbb{Z}$-action with zero topological entropy.} action with zero topological entropy. It is not even known whether every system with zero entropy (even zero-dimensional) admits a symbolic extension.
This is why the Symbolic Extension Entropy Theorem was formulated in \cite{DZ} for general amenable group actions in a slightly weaker version, which we call
Quasi-Symbolic Extension Entropy Theorem, in which symbolic extensions are replaced by quasi-symbolic extensions. Some sufficient conditions under which quasi-symbolic extensions can be replaced by genuine symbolic extensions were discussed in \cite{DZ}.

\section{Basics for amenable group actions} \label{amenable}

In this section we will make some preparations for discussing the case of amenable group actions, in particular we present elements of the entropy theory and theory of entropy structure for such actions.

\medskip

Amenability of a group was introduced by Neumann \cite{vN}. Nowadays there are many equivalent ways of defining amenability and most of them apply to groups much more general than countable (see for example \cite{P, Pi}). In this survey we will use the one which fits us best, which relies on the concept of a F\o lner \sq\ introduced by F\o lner \cite{Fo}.
We will use $|F|$ to denote the cardinality of a set $F$. Given a finite\footnote{Here by a finite set we are always meaning a nonempty finite set.} set $F\subset G$ and $\varepsilon>0$, an \emph{$\varepsilon$-modification} of $F$ is any set $F'$ such that $\frac{|F\triangle F'|}{|F|}<\varepsilon$, where $\triangle$ denotes the symmetric difference of sets. An $\varepsilon$-modification of $F$ which is also a subset of $F$ will be called a \emph{$(1- \varepsilon)$-subset} of $F$. If $K$ is another finite subset of $G$ then $F$ is
called \emph{$(K, \varepsilon)$-invariant} if $K F$ is an $\varepsilon$-modification of $F$, where $K F= \{k f: k\in K, f\in F\}$.

\begin{definition}\label{tyi}
A \sq\ $(F_n)_{n\in\N}$ of finite subsets of $G$ is called a \emph{F\o lner \sq} if, for every finite set $K$ and every $\varepsilon>0$, the sets $F_n$ are eventually (i.e.\  except for finitely many indices $n$) $(K,\varepsilon)$-\inv. A countable group which possesses a F\o lner \sq\ is called \emph{amenable}.
\end{definition}

The class of  discrete countable amenable groups includes all Abelian, nilpotent, and more generally, solvable groups. We remark that any countable amenable group possesses a F\o lner \sq\ $(F_n)_{n\in\N}$ with the following additional properties\footnote{In the special case of $G=\Z$ we can take $F_n=\{- n, - n+ 1, \dots, n\}$ for each $n\in \N$.}, where the first two properties are easily obtained and for the existence of symmetric F\o lner \sq s see \cite[Corollary 5.3]{N}:
\begin{enumerate}

\item \emph{centered}: $\forall_{n\in\N}\ e\in F_n$.

\item \emph{nested}: $\forall_{n\in\N}\ F_n\subset F_{n+1}$.

\item \emph{symmetric}: $\forall_{n\in\N}\ F_n=F_n^{-1}$ (by convention, $F_n^{-1}=\{f^{-1}:f\in F_n\}$).
\end{enumerate}

\medskip
For a $G$-action $(X, G)$ of a countable amenable group we have well-defined notions of measure-theoretic entropy $h_\mu (X, G)$ (also denoted by $h_\mu (X)$) for every $\mu\in \MGX$ and topological entropy $\htop (X,G)$, see for example \cite{KLbook, MO85, OW}. Let us briefly recall the basics.

 Let $\P$ be a finite Borel partition of $X$ and $\mu\in \MGX$. The \emph{Shannon entropy of $\P$ with respect to $\mu$} equals
$$
H(\mu,\P)=-\sum_{P\in\P}\mu(P)\log(\mu(P))\le\log|\P|.
$$
Now, given a finite set $F\subset G$, by $\P^F$ we mean the join partition generated by $g^{- 1} (\P), g\in F$, that is,
$$
\P^F=\bigvee_{g\in F}g^{-1}(\P)=\Bigl\{\bigcap_{g\in F}g^{-1}(P_g): \forall_{g\in F}\ P_g\in\P\Bigr\}.
$$
One of elementary properties of the Shannon entropy, is the following subadditivity property:
$$\forall\ \text{finite}\ F_1, F_2\subset G, H(\mu,\P^{F_1\cup F_2})\le H(\mu,\P^{F_1})+H(\mu,\P^{F_2}).
$$
It can be proved that, for any given F\o lner \sq\ $(F_n)_{n\in\N}$ in $G$, the limit $\lim_{n\rightarrow \infty} \frac1{|F_n|}H(\mu,\P^{F_n})$ exists and its value is independent of the selection of the F\o lner \sq. This is due to the well-known Ornstein--Weiss Lemma proved via the standard quasitiling machinery built by Ornstein and Weiss for countable amenable groups (see for example \cite{LW00, OW}, see also \cite{Gr, HYZ, WZ}). Thus one can define, given any F\o lner \sq\ $(F_n)_{n\in\N}$ in $G$, the \emph{dynamical entropy of $\P$ with respect to $\mu$} as
$$
h_\mu (G, \P) = \lim_{n\rightarrow \infty} \frac1{|F_n|}H(\mu,\P^{F_n})
$$
(if the action of $G$ is understood, we will simply write $h_\mu (\P)$; similar convention will be used later for conditional entropy).
In fact, strong subadditivity holds for the Shannon entropy (see for example \cite{DFR, MO85})
$$
H(\mu,\P^{F_1\cup F_2})\le H(\mu,\P^{F_1})+H(\mu,\P^{F_2})-H(\mu,\P^{F_1\cap F_2}),
$$
and then one can prove that the limit defining the dynamical entropy of a partition equals
$$
\inf \left\{\frac1{|F|}H(\mu,\P^F): F\ \text{is a finite subset of}\ G\right\}.
$$
Now the \emph{Kolmogorov--Sinai entropy of $\mu$} with respect to the action $(X, G)$ is defined as
$$
h_\mu (X, G)=\sup \{h_\mu (\P): \P\ \text{is a finite Borel partition of}\ X\}.
$$
In particular, the Kolmogorov--Sinai entropy of $\mu$ does not depend on the choice of the F\o lner \sq\ in $G$, and
if $(\P_k)_{k\in\N}$ is a refining \sq\ of finite Borel partitions then
$$
h_\mu (X, G)=\lim_{k\rightarrow \infty} \uparrow h_\mu (\P_k).
$$
Sometimes $h_\mu (X, G)$ will also be denoted shortly by $h(\mu)$ (when no confusion occurs) and the function $\mu\mapsto h(\mu)$ on $\M (X, G)$ will be called the \emph{entropy function} of $(X, G)$.

Now consider two finite Borel partitions of $X$, $\P$ and $\Q$. In this context one defines the \emph{conditional Shannon entropy of $\P$ given $\Q$} (with respect to $\mu$) as
$$
H(\mu,\P|\Q) = H(\mu,\P\vee\Q)-H(\mu,\Q).
$$
The \emph{conditional dynamical entropy of $\P$ given $\Q$}, with respect to $\mu$ is defined analogously, as
$$
h_\mu (G, \P|\Q) = h_\mu(\P|\Q) = \lim_{n\rightarrow \infty} \frac1{|F_n|}H(\mu,\P^{F_n}|\Q^{F_n}) = h_\mu (\P\vee\Q)-h_\mu (\P).
$$
For conditional Shannon entropy subadditivity still holds:
$$
H(\mu,\P^{F_1\cup F_2}|\Q^{F_1\cup F_2})\le H(\mu,\P^{F_1}|\Q^{F_1})+H(\mu,\P^{F_2}|\Q^{F_2}),
$$
however the strong subadditivity in general fails.
Even so, the limit in defining the conditional dynamical entropy still can be replaced by the infimum over all finite sets $F\subset G$. In fact, if
$\Q^G$ denotes the smallest
	\inv\ $\sigma$-algebra generated by $\Q$, then
$$h_\mu (\P|\Q) = \lim_{n\rightarrow \infty} \frac1{|F_n|}H(\mu,\P^{F_n}|\Q^G)= \inf_F \frac1{|F|}H(\mu,\P^{F}|\Q^G)= \inf_F
	\frac1{|F|}H(\mu,\P^F|\Q^F),$$
where both infima are taken over all finite subsets $F\subset G$
(for the proofs see for example \cite{DooZ12, DooZ15, GTW, WZ}). As one of important consequences one has the following observation, which will be useful in our discussion about entropy structure for amenable group actions.

\begin{proposition}\label{dodado}
Let $\P$ and $\Q$ be both finite Borel partitions of $X$.
\begin{enumerate}

\item Then the function $\mu\mapsto h_\mu (\P|\Q)$ is affine on $\MGX$, and hence the entropy function of $(X, G)$ is also affine on $\MGX$.

\item Assume that the boundary of any atom of both $\P$ and $\Q$ has measure zero for every  $\mu\in\MGX$. Then the function $\mu\mapsto h_\mu (\P|\Q)$ is u.s.c.\ on $\MGX$.
\end{enumerate}
\end{proposition}

Since \tl\ entropy plays in this survey only a marginal role, we reduce its presentation to a necessary minimum.
We are reluctant to present a lengthy original definition, alternatively,
 we use the following understanding of \tl\ entropy (via the variational principle for amenable group actions \cite{MP82, ST80}).

\begin{definition}
The \emph{\tl\ entropy} of the action $(X,G)$ equals
$$
\htop(X,G)=\sup_{\mu\in\MGX}h_\mu(X, G).
$$
\end{definition}

We can introduce the notion of a principal extension for amenable group actions following Ledrappier's idea \cite{led}. Let $\pi:Y\to X$ be an extension from a $G$-action $(Y, G)$ to another $G$-action $(X, G)$. Then, for every $\nu\in\MGY$, one can define the \emph{conditional entropy of $\nu$ given $X$}, as
$$
h_\nu (G,Y| X)=h_\nu (Y| X)= \sup_\Q \inf_\P h_\nu (G, \Q| \pi^{- 1} (\P)),
$$
where $\Q$ ranges over all finite Borel partitions of $Y$, while $\P$ ranges over all finite Borel partitions of $X$.
We say that $\pi$ is a \emph{principal extension} of $X$, if $h_\nu (Y|X)=0$ for every $\nu\in\MGY$. In this case, $Y$ is also called a principal extension of $X$. By direct calculations it is not hard to obtain that if $h_\mu (X,G)<\infty$, where $\mu= \pi(\nu)\in\MGX$, then $h_\nu (Y|X)$ is simply the difference $h_\nu (Y, G)-h_\mu (X, G)$. Thus the definition of a principal extension introduced here coincides with Definition \ref{prin} in the setting of $\mathbb{Z}$-actions.

Entropy structure of an amenable group action \xg\ is defined by exactly the same two steps as it was done in the case of $\Z$-actions.  This is possible
thanks to a deep result by Huczek, \cite[Theorem 2]{H}, which assures that any $G$-action admits a \zd\ principal extension. We also remark that for amenable group actions on \zd\ spaces a result analogous to Proposition \ref{202002121502} holds. However, it is not known whether a result about the small boundary property similar to that in \cite{L99, LW00} holds for amenable group actions, even though the small boundary property has been studied for $\mathbb{Z}^k$-actions in \cite{GLT} and for amenable group actions in \cite{KS, LT} (but we are not going to need this).

By the same proof as for $\Z$-actions one derives, using Proposition \ref{dodado}, the following result:

\begin{theorem} \label{202002161830}
Any $G$-action with finite topological entropy admits an affine \ens\ with u.s.c.\ differences.
\end{theorem}

\section{Quasitilings, tilings and tiling systems of amenable groups} \label{prepa}

For a long time the standard machinery of quasitilings developed by Ornstein and Weiss \cite{OW} has played a crucial role in ergodic and entropy theories of amenable group actions. But \qt s turn out to be not suited for building symbolic extensions. The part of the group that is not covered by the quasi-tiles even though it is a very small set (of small Banach density) presents a serious obstacle in the construction. In fact, as explained in \cite{DZ}, one needs more precise tilings, which cover all of the group. During the long process of building up the foundations for the symbolic extension theory of amenable group actions, it was proved in \cite{DHZ} that the Ornstein--Weiss \qt s can be improved to become tilings.

In this section we will recall results from \cite{DH, DHZ, DZ} about quasitilings, tilings and tiling systems, all of which will play fundamental roles in the construction of symbolic extensions (in fact quasi-symbolic extensions) for amenable group actions.

\medskip

\begin{definition}
A \emph{\qt} of $G$ is a countable family $\CT$ of finite sets $T\subset G$, called the \emph{tiles}, together with a map from $\CT$ to $G$ assigning to each tile $T$ a point $c_T\in T$ called the \emph{center} of $T$. We assume that the above map is injective, i.e.\ that different tiles have different centers. The image of this injection, i.e.\  the set $C(\CT)=\{c_T:T\in\CT\}$ will be referred to as the \emph{set of centers} of $\CT$.
For each tile $T$, the set $S_T=Tc_T^{-1}$ will be called the \emph{shape} of $T$ (we remark that every shape contains the unity $e$). The \emph{collection of shapes} $\{S_T:T\in\CT\}$ will be denoted by $\CS(\CT)$. Given $S\in\CS(\CT)$, the set $C_S=C_S(\CT)=\{c_T\in C(\CT):S_T=S\}$ will be called the \emph{set of centers for the shape $S$}.
\end{definition}

 Note that the sets of centers for different shapes of a \qt\ $\CT$
 are disjoint and their union over all shapes equals $C(\CT)$. A \qt\ $\CT$ is \emph{proper} if the collection of shapes $\CS(\CT)$ is finite. From now on, \emph{by a \qt\ we shall always mean a proper \qt.}

\begin{definition}\label{qt} Let $K\subset G$ be a finite set and fix an $\eps\in[0,1)$. A \qt\ $\CT$ is called
\begin{enumerate}

	\item \emph{$(K,\eps)$-invariant} if all shapes $S\in \CS(\CT)$ are \emph{$(K,\eps)$-invariant}.

	\item \emph{$\eps$-disjoint} if there exists a mapping $T\mapsto T^\circ$ (\,$T\in\CT$) such that,
every $T^\circ$ is a $(1\!-\!\eps)$-subset of $T$ and
the family $\{T^\circ:T\in\CT\}$ is disjoint.

    \item \emph{disjoint} if the tiles of $\CT$ are pairwise disjoint.

	\item \emph{$(1- \eps)$-covering} if $\underline D(\bigcup\CT)\ge 1- \eps$, where
the \emph{lower Banach density} of a set $A\subset G$ is defined by the formula
$$\underline D(A)= \sup \left\{\inf_{g\in G} \frac{|A\cap F g|}{|F|}: F\subset G\ \text{is finite}\right\}.$$

	\item a \emph{tiling} if it is a partition of $G$.
\end{enumerate}
By an \emph{$\varepsilon$-\qt} we shall mean a \qt\ which is both $\varepsilon$-disjoint and $(1-\varepsilon)$-covering.
\end{definition}

\emph{From now on in the group $G$ we fix a symmetric, centered F\o lner \sq\ $(F_n)_{n\in\N}$.}

The existence of  \qt s with additional properties (listed below) is proved in \cite[Lemma 4.1]{DHZ}. The proof is essentially the same as in \cite{OW}, the differences result from using the notion of lower Banach density for the covering parameter of a \qt.

\begin{proposition}\label{OW}
For any $\varepsilon>0$ there exists $r(\varepsilon)\in \N$ such that for any $n\in \N$ there exists
an $\varepsilon$-\qt\ $\CT$ of $G$ with $\CS(\CT)\subset\{F_{n_1}, F_{n_2}, \dots, F_{n_{r(\varepsilon)}}\}$, where $n< n_1<n_2<\cdots<n_{r(\varepsilon)}$.
\end{proposition}

In the sequel we shall use the notion of a dynamical \qt. For better differentiation of the notions, the \qt s defined so far will be referred to as \emph{static}. A static \qt\ $\CT$ can be identified as a point of the symbolic space ${\rm V}^G$ where ${\rm V}=
\{``S\,":S\in\CS(\CT)\}\cup\{0\}$. Namely, for each $S\in\CS(\CT)$ we place the symbol $``S"$ at all the centers $c\in C_S$, and we place the symbol $``0"$ at all remaining positions. Formally, we can write $\CT=\{\CT_g:g\in G\}$, where
$$
\CT_g=\begin{cases}``S"\,;& g\in C_S, S\in\CS(\CT),\\
``0";& g\notin C(\CT).
\end{cases}
$$

\begin{definition} By a \emph{dynamical \qt} with the finite collection of shapes $\CS$ we will understand any subshift $\T\subset {\rm V}^G$, where ${\rm V}=\{``S\,":S\in\CS\}\cup\{``0"\}$ (the elements of $\T$ are interpreted as static \qt s $\CT$ with $\CS(\CT)\subset\CS$). The set $\CS$ will be also denoted as $\CS(\T)$. We will say that the dynamical \qt\ $\T$ is \emph{$(K,\varepsilon)$-\inv}, \emph{$\varepsilon$-disjoint}, \emph{disjoint}, \emph{$(1- \varepsilon)$-covering}, a \emph{dynamical $\varepsilon$-\qt} or a \emph{dynamical tiling} if all its elements are $(K,\varepsilon)$-\inv, $\varepsilon$-disjoint, disjoint, $(1- \varepsilon)$-covering, $\varepsilon$-\qt s or tilings, respectively.
\end{definition}

By combining the results of \cite{DH, DHZ}, we have at our disposal the following general fact:

\begin{proposition}\label{ext}
For any $\varepsilon>0$ and any finite set $K\subset G$, there exists a $(K,\varepsilon)$-\inv\ dynamical tiling $\T$ of $G$ with zero topological entropy.
\end{proposition}

However, one dynamical \qt\ (or tiling) is insufficient for later construction of a quasi-symbolic extension. What we need is a countable joining of a \sq\ of dynamical \qt s (tilings), with improving disjointness, covering and invariance properties.

By a \emph{\tl\ joining} of finitely or countably many systems $X_k$ ($k\in K$ where $K=\{1,2,\dots,l\}$ with $l\in\N$, or $K=\N$) we will mean any closed subset $Z$ of the Cartesian product $\prod_{k\in K}X_k$, which is invariant under the product (coordinatewise) action, and whose projection on every coordinate is surjective. Such a joining will be sometimes denoted by $\bigvee_{k\in K}X_k$ (although this notation is ambiguous, as there may exist many joinings of the same collection of systems). At least one joining always exists---the product joining.

\begin{definition}
Let $(\eps_k)_{k\in\N}$ be a decreasing to $0$ \sq\ of positive numbers. Let $\mathbf T=\bigvee_{k\in\N}\T_k$ be a \tl\ joining of a \sq\ of dynamical \qt s of $G$, such that for every $k\in\N$, $\T_k$ is a (dynamical) $\eps_k$-\qt, and for every $\varepsilon>0$ and finite set $K\subset G$, for $k$ sufficiently large all shapes of $\T_k$ are $(K,\varepsilon)$-\inv. Such $\mathbf T$ will be called a \emph{F\o lner system of \qt s}.\footnote{The term ``F\o lner system of \qt s'' comes from the fact that it is a \tl\ action and from the
observation that the last requirement in the definition can be formulated differently: the joint collection $\bigcup_{k\in\N}\CS_k$, indexed (bijectively, but in an arbitrary order) by natural numbers, is a F\o lner \sq\ in $G$.} The elements of $\mathbf T$ will be denoted by $\boldsymbol\CT=(\CT_k)_{k\in\N}$ ($\forall_{k\in\N}\,\CT_k\in\T_k$). The collection of shapes of $\T_k$, $\CS(\T_k)$, will be abbreviated as $\CS_k$.
\end{definition}

The existence of a F\o lner system of \qt s for the group $G$ follows directly from Proposition \ref{OW}. We remark that in fact one can find such a system as a \tl\ factor of any given \zd\ free $G$-action \cite{DZ}. We can deduce similarly the existence of F\o lner systems for tilings by using alternatively Proposition \ref{ext}, but we cannot claim that given a free \zd\ $G$-action \xg, a F\o lner system of tilings (or even one dynamical tiling) appears as a \tl\ factor of \xg.

In case of a F\o lner system of tilings we can further demand the members of the joining to ``interact'' with each other in a more specific manner, which is crucial in our later construction.

Two key such interactions are congruency and determinism, as defined below:

\begin{definition}
Let $\mathbf T=\bigvee_{k\in\N}\T_k$ be a F\o lner system of tilings. We say that $\mathbf T$ is
\begin{enumerate}

	\item \emph{congruent} if $\forall \boldsymbol{\CT}=(\CT_k)_{k\in\N}\in\mathbf T$ every tile of $\CT_{k+1}$  is a union of some tiles of~$\CT_k$ ($k\in\N$).

	\item \emph{deterministic}\footnote{In fact any congruent F\o lner system of tilings $\mathbf T=(\T_k)_{k\in\N}$ can be easily made deterministic in an inductive process of duplicating the shapes, for details see \cite{DH}.} if it is congruent, and additionally, for each $k\in\N$ and every shape $S'\in\CS_{k+1}$,
	there exist sets $C_S(S')\subset S'$ ($S\in\CS_k$) such that
	$$
	S'= \bigsqcup_{S\in\CS_k}\ \bigsqcup_{c\in C_S(S')} Sc \text{\ \ (which is a disjoint union),}
	$$
	and for each $\boldsymbol{\CT}=(\CT_k)_{k\in\N}\in\mathbf T$, whenever $S'c'$ is
	a tile of $\CT_{k+1}$ then the sets $Scc'$ with $S\in\CS_k$ and $c\in C_S(S')$ are tiles
	of $\CT_k$. We also define $C_k(S')=\bigcup_{S\in\CS_k}C_S(S')$.

\end{enumerate}
\end{definition}

In the deterministic case, each static tiling $\CT_{k+1}$ \emph{determines} the tiling
$\mathcal T_k$ joined with it, because each of the tiles of $\CT_{k+1}$ is partitioned into the tiles of $\CT_k$ in a unique way determined by its shape. In fact in such case the joining $\mathbf T$ is an \emph{inverse limit}\footnote{We assume that $(X_k)_{k\in\N}$ is a \sq\ of systems such that, for each $k\in\N$, $X_k$ is a factor of $X_{k+1}$ via a map $\pi_k$. Then the \emph{inverse limit} of the \sq\ $(X_k)_{k\in\N}$ is defined as
$\overset\leftarrow{\lim\limits_k}\, X_k=\{(x_k)_{k\in\N}: \forall_{k\in\N}\ x_k\in X_k \text{ and }x_k=\pi_k(x_{k+1})\}$, where the action is defined coordinatewise.} denoted by
$\mathbf T=\overset\leftarrow{\lim\limits_k}\T_k$.
\begin{definition}
A deterministic F\o lner system of tilings will be shortly called a \emph{tiling system}.
\end{definition}

The following result is \cite[Theorem 5.2]{DHZ} translated to the terminology introduced above. We remark that congruency has been required explicitly in \cite{DHZ}, while determinism is implicit in the proof. In order to obtain \tl\ entropy zero, the technique of turning a congruent system of tilings into a deterministic one by duplicating the shapes, is insufficient. In the proof, determinism is obtained by a more sophisticated method.

\begin{theorem}\label{fs}
For any infinitely countable amenable group $G$ there exists a tiling system (i.e.\ a deterministic F\o lner system of tilings) of $G$ with zero \tl\ entropy.
\end{theorem}

An important element of the construction of symbolic extensions for amenable group actions (as well as for $\Z$-actions) is a special integer-valued function on rectangular symbolic blocks appearing in the array representation of a \zd\ system, called an \emph{oracle}. While the creation of an oracle for $\Z$-actions was possible using standard properties of entropy, in the general amenable case some serious technical difficulties have been encountered, due to lack of subadditivity of certain conditional entropy functions. This problem was resolved in \cite{DZ} using the so-called \emph{tiled entropy}. Tiled entropy is a technical invariant with a lengthy definition and in general cannot be easily reduced to a standard notion of conditional entropy. In this survey we are going to devote some space to describing briefly this new notion, because it provides a new approach to dynamical entropy and may be of independent interest. Also, it is very natural in the context of tiling systems. The remainder of this section will be devoted to tiled entropy.
\medskip

We assume that $\mathbf T=\overset\leftarrow{\lim\limits_k}\T_k$ is a tiling system of $G$ with \tl\ entropy zero (we know that such a tiling system exists for every countable amenable group $G$). Recall that the set of shapes of $\T_k$ is denoted by $\CS_k$, and given $\boldsymbol\CT=(\CT_k)_{k\in\N}\in\mathbf T$, the set of centers of $\CT_k$ of the tiles with shape $S\in\CS_k$ is denoted by $C_S(\CT_k)$. The set of all centers of $\CT_k$ is $C(\CT_k)=\bigcup_{S\in\CS_k}C_S(\CT_k)$. We now introduce more notation. For $S\in\CS_k$, and $s\in S$ by $[S,s]$ we denote the set of elements $\boldsymbol\CT\in\mathbf T$ for which $s^{-1}$ belongs to $C_S(\CT_k)$. If $\boldsymbol\CT\in[S,s]$ then $Ss^{-1}$ is the tile of $\CT_k$ which contains the unity, i.e.\  \emph{the central tile of $\CT_k$}. The set $[S,e]$ will be abbreviated as $[S]$.
Observe that  $\boldsymbol\CT\in[S]$ if and only if $\CT_{k,e}=``S\,"$, so the notation is consistent with that of one-symbol cylinders over the alphabet ${\rm V}_k=\{``S\,": S\in\CS_k\}\cup\{``0"\}$.
The family $\D_{\CS_k}=\{[S,s]:S\in\CS_k,\ s\in S\}$ is a partition of $\mathbf T$. Also note that $\boldsymbol\CT\in[S,s]$ if and only if $s^{-1}(\boldsymbol\CT)\in[S]$, i.e.\  $[S,s]=s([S])$. So, if $\nu$ is a shift-\im\ on $\mathbf T$, then $\nu([S,s])=\nu([S])$ for all $s\in S$.
Recall, that by  determinism of the tiling system, whenever $k'>k$, every shape $S'$ of $\T_{k'}$ decomposes in a unique way as a concatenation of shifted shapes of $\T_k$ and the set of centers of these tiles is denoted by $C_k(S')$. The subset of $C_k(S')$ consisting of centers of tiles with a particular shape $S\in\CS_k$ is denoted by $C_S(S')$ (now the subscript $k$ is not needed; $k$ is determined by $S$).

Now let $(X, G)$ be a \zd\ action, given in its \emph{array representation} $X=\overset\leftarrow{\lim\limits_k} X_{[1,k]}$, which can be introduced similarly as in \S \ref{four}. Let us assume additionally that $X$ has the tiling system $\mathbf T$ as a \tl\ factor (in general we can replace $X$ by its joining with $\mathbf T$, so this assumption will be fulfilled naturally). Then we can combine the layers of $X$ (i.e.\  the subshifts $X_k$) with the layers of $\mathbf T$ (i.e.\  the dynamical tilings $\T_k$) by replacing each $X_k$ by its \tl\ joining $\bar X_k$ with $\T_k$ realized naturally in the common extension $X$. The combined alphabet of $\bar X_k$ is $\bar\Lambda_k=\Lambda_k\times \rm V_k$, while that of $\bar X_{[1,k]}$ is $\bar\Lambda_{[1,k]}=\prod_{l=1}^k\bar\Lambda_l$.  In this manner, the system on $X$ is replaced by its \tl ly conjugate model which is the inverse limit $\bar X = \overset\leftarrow{\lim\limits_k} \bar X_{[1,k]}$. We will call $\bar X$ the \emph{tiled array representation of $X$}.

We will use the following notational convention. For $\bar x\in\bar X$ and $S\in\CS_k$, the expression $\bar x_g=``S\,"$ means that in the tiling $\CT_k$ apparent in the $k$th layer of $\bar x$, at the position $g$ there occurs a center of a tile with shape $S$. Formally, this means that if $\bar x_k$ is the $k$th layer of $\bar x$ then $\bar x_{k,g} \in\Lambda_k\times\{``S\,"\}\subset\bar\Lambda_k$.
With this convention, the notions $[S]$ and $[S,s]$ ($S\in\CS_k,\ s\in S$) may be applied to $\bar X$ in the following way:
$$
[S]=\{\bar x\in \bar X: \bar x_e=``S\,"\}, \ \ [S,s]=s([S]) =\{\bar x\in \bar X: \bar x_{s^{-1}}=``S\,"\}.
$$

\begin{definition}\footnote{Alternatively, one can define just the unconditional version and then put
$$
H_{\T_k}(\mu,\P|\Q)=H_{\T_k}(\mu,\P\vee\Q)-H_{\T_k}(\mu,\Q).
$$
} Let $\P$ and $\Q$ be two finite Borel partitions of $\bar X$ with $|\P|>1$, and $\mu$ a Borel probability measure on $\bar X$. By the \emph{$k$th tiled entropy of $\P$} and \emph{$k$th conditional tiled entropy of $\P$ given $\Q$} with respect to $\mu$ we will mean the following terms:
$$
H_{\T_k}(\mu,\P)=\sum_{S\in\CS_k}\mu([S])H(\mu_{[S]},\P^S), \ \ H_{\T_k}(\mu,\P|\Q)=\sum_{S\in\CS_k}\mu([S])H(\mu_{[S]},\P^S|\Q^S),
$$
where $\mu_{[S]}$ is the normalized conditional measure $\mu$ on $[S]$.
\end{definition}

The monotonicity of the tiled entropy revealed below, in \cite[Theorem 4.27]{DZ}, will be crucial in the construction of quasi-symbolic extensions in the Quasi-Symbolic Extension Entropy Theorem. We remark that the standard static entropy (with respect to elements $F_n$ of a F\o lner \sq) lacks subbadditivity and thus, in general, fails to behave monotonically with respect to $n$.

\begin{theorem}\label{tte}
The \sq s of tiled entropies $(H_{\T_k}(\mu,\P))_{k\in\N}$ and $(H_{\T_k}(\mu,\P|\Q))_{k\in\N}$, converge on $\M (\bar X, G)$ decreasingly to $h_\mu (G,\P)$ and $h_\mu (G,\P|\Q)$, respectively.
\end{theorem}

\section{The symbolic extension theory for amenable group actions} \label{theory}

 In this section, we introduce the Quasi-Symbolic Extension Entropy Theorem obtained in \cite{DZ} for general amenable group actions. Here, by a quasi-symbolic system, we mean a \tl\ joining of a subshift with a zero entropy tiling system discussed in the previous section.

 \medskip

In order to answer Question \ref{202002081406} and compute $\mathbf{h_\text{sex}}$ efficiently for $G$-actions,
it is natural to expect that a result similar to Theorem \ref{sex-vp} holds. The definitions of \emph{extension entropy function} and the \emph{symbolic extension entropy function} in Definition \ref{202002092127} work directly for $G$-actions. We will use the same notations for these functions.
For $G$-actions one inequality in the Symbolic Extension Entropy Variational Principle holds trivially (see Proposition \ref{inequa}).
While, as shown in \cite{BD, D1}, the reversed inequality in Theorem~\ref{sex-vp} depends essentially on Theorem~\ref{seet}, the Symbolic Extension Entropy Theorem. As we shall see, the validity of similar results for $G$-actions will face some limitations, and hence the reversed direction of Proposition \ref{inequa} will also face the same limitations.

\begin{proposition} \label{inequa}
$\mathbf{h_\text{sex}} (X, G)\ge \sup \{h_\text{sex} (\mu): \mu\in \mathcal{M} (X, G)\}$.
\end{proposition}

As we can still talk about \ens\ and its \se\ for $G$-actions with finite topological entropy, inspired by Theorem \ref{seet} one may ask:

\begin{question} \label{ask-amen}
Let $\mathcal{H}$ be an \ens\ of an action $(X, G)$ with finite topological entropy, and let $E$ be a bounded function on $\M(X,G)$. Is it true that $E$ equals $h^\pi$ in some symbolic extension $\pi:Y\to X$ if and only if it is an affine superenvelope of $\H$?
\end{question}

We can immediately give an affirmative answer concerning the implication $\Longrightarrow$:

\begin{theorem}\label{easy}
Let $\pi: Y\to X$ be a
symbolic extension of the action $(X, G)$. Then the extension entropy function $h^\pi$ is an affine \se\ of the \ens\ of $X$.
\end{theorem}

The aforementioned quasi-symbolic extensions are defined as follows.

\begin{definition}
By a \emph{quasi-symbolic system} $\bar Y$ we mean a \tl\ joining $Y\vee\mathbf T$ of a subshift $Y$ with a zero entropy tiling system $\mathbf T$. By a \emph{quasi-symbolic extension} of a system we mean a \tl\ extension which is a quasi-symbolic system.
\end{definition}

With this notion in hand, we can prove (in full generality) a slightly weakened version of the Symbolic Extension Entropy Theorem, where the symbolic extensions are replaced by quasi-symbolic extensions:

\begin{theorem}\label{quasi}
Let $\mathcal{H}$ be an \ens\ of the action $(X, G)$ with finite topological entropy. Then $\EA$ is a (finite) affine \se\ of the structure $\H$ if and only if there exists a quasi-symbolic extension $\pi:\bar Y\to X$ such that $\EA=h^\pi$.
\end{theorem}

We remark that it is unknown whether, in full generality, the quasi-symbolic extension can be repaced by a symbolic one. This depends on the unsolved question, does the zero-entropy tiling system admit a zero entropy symbolic extension (in fact it is unknown whether it admits any symbolic extension). This problem, in turn, depends on another, concerning the so-called comparison property. We will discuss this problem in Section \ref{im-theory}.

\medskip

With the help of Theorem \ref{quasi}, the following result is obtained in \cite{DZ}.

\begin{theorem} \label{ahe-am}
Let $(X, G)$ be a $G$-action with finite topological entropy. Then the following statements are equivalent:
 \begin{enumerate}


 \item \label{sec-am} The action $(X, G)$ admits an \ens\ such that it converges uniformly to the entropy function $\mu\mapsto h_\mu (X, G)$ on $\mathcal{M} (X, G)$.

 \item \label{th-am} The action $(X, G)$ admits a principal quasi-symbolic extension.


\item The entropy function on $\mathcal{M} (X, G)$ is the minimal superenvelope of the entropy structure of $(X, G)$.
 \end{enumerate}
\end{theorem}

We also remark that, if $G$ contains $\mathbb{Z}$ as a subgroup of infinite index, then, by \cite[Lemma 5.7]{LThom}, any $C^1$ action $(X, G)$ on a compact smooth manifold $X$ has zero topological entropy, and so $(X, G)$ admits a principal quasi-symbolic extension by Theorem \ref{ahe-am}.

\begin{remark} \label{remark}
As we shall see in the next section, for some types of the group $G$ we can replace the quasi-symbolic extensions in Theorem \ref{quasi} by genuine symbolic extensions. For the same groups we can replace quasi-symbolic extensions by symbolic extensions in Theorem \ref{ahe-am}.
\end{remark}

\section{Improvements of Theorem \ref{quasi}} \label{im-theory}

The quasi-symbolic extensions in Theorem \ref{quasi} are joinings of symbolic systems with a zero entropy tiling system discussed in \S \ref{prepa}. In general, the tiling system is just a \zd\ system, not a subshift. It is hence clear that the quasi-symbolic extensions in Theorem \ref{quasi} can be replaced by symbolic extensions for such amenable groups, for which there exists a zero-entropy tiling system which admits a zero entropy symbolic extension.
In this section we will present some classes of such groups.

\medskip
Recall that a countable group is \emph{residually finite} (see for example \cite{Ma}) if the intersection of all its subgroups of finite index is the trivial group $\{e\}$, equivalently, the intersection of all its normal subgroups of finite index is trivial. For countable amenable groups which are residually finite, the full version of the Symbolic Extension Entropy Theorem holds:

\begin{theorem}\label{fc}
Let $\mathcal{H}$ be an \ens\ of an action $(X, G)$, where the group $G$ is countable amenable and residually finite, and the action $(X, G)$ has finite topological entropy. Then $\EA$ is a (finite) affine \se\ of the structure $\H$ if and only if there exists a symbolic extension $\pi: Y\to X$ such that $\EA=h^\pi$.
\end{theorem}

The second class of groups relies on the so-called \emph{comparison property}. Comparison originates in the theory of \mbox{$C^*$-al}geb\-ras, but the most important for us ``dynamical'' version concerns group actions on compact metric spaces. In this setup it was defined by Cuntz \cite{Cu} and further investigated by R\o rdam \cite{MR1,MR2} and by Winter \cite{W}. As in the case of many other properties and notions in dynamical systems, the most fundamental form of comparison occurs in $\Z$-actions. For a wider generality, we refer the reader to a recent paper by Kerr \cite{K}, where the notion is defined for other actions including \tl\ and measure-preserving ones. We will focus on a particular case where a countable amenable group acts on a \zd\ compact metric space.

\begin{definition}\label{defcom}
Recall that $G$ is assumed to be a discrete countably infinite amenable group, and let $(X, G)$ be a \zd\ action. We say that
\begin{enumerate}

\item $(X, G)$ \emph{admits comparison} if for any pair of clopen sets $\mathsf A,\mathsf B\subset X$, the condition that $\mu(\mathsf A)<\mu(\mathsf B), \forall \mu\in \mathcal{M} (X, G)$, implies that $\mathsf A$ is \emph{subequivalent} to $\mathsf B$, denoted by $\mathsf A\preccurlyeq \mathsf B$, that is, there exists a finite partition $\mathsf A=\bigcup_{i=1}^k \mathsf A_i$ of $\mathsf A$ into clopen sets and there are elements $g_1, \dots, g_k$ of $G$ such that $g_1(\mathsf A_1), \dots, g_k(\mathsf A_k)$ are disjoint subsets of $\mathsf B$.\footnote{Clearly, $\mathsf A\preccurlyeq \mathsf B$ implies $\mu(\mathsf A)\le\mu(\mathsf B), \forall \mu\in \mathcal{M} (X, G)$, and so in some sense the comparison is ``nearly'' an equivalence between subequivalence and the inequality for all \im s.}

\item $G$ has the \emph{comparison property} if all \zd\ $G$-actions admit comparison.
\end{enumerate}
\end{definition}

\begin{remark}
As shown in \cite{DZ}, the action $(X, G)$ admits comparison if and only if it admits \emph{weak comparison}, that is, there exists a constant $C\ge 1$ such that for any clopen sets $\mathsf A,\mathsf B\subset X$, the condition $\sup \{\mu(\mathsf A): \mu\in \mathcal{M} (X, G)\}<\frac1C\inf \{\mu(\mathsf B): \mu\in \mathcal{M} (X, G)\}$ implies $\mathsf A\preccurlyeq \mathsf B$.
\end{remark}

The following results establish a strong connection between comparison and the existence of tiling systems with special properties.

\begin{theorem} \label{dt}
Let $(X, G)$ be a free \zd\ action. Then the action admits comparison  if and only if it has a tiling system as a \tl\ factor.
\end{theorem}

\begin{theorem}\label{ddt}
Assume that $G$ has the comparison property. Then there exists a zero entropy tiling system of $G$, such that the system admits a principal symbolic extension.
\end{theorem}

Directly from Theorems \ref{quasi} and \ref{ddt}, one derives the full version of the Symbolic Extension Entropy Theorem for groups with the comparison property:

\begin{theorem}\label{coc}
Let $\mathcal{H}$ be an \ens\ of an action $(X, G)$, where the group $G$ has the comparison property and the action $(X, G)$ has finite topological entropy. Then $\EA$ is a (finite) affine \se\ of the structure $\H$ if and only if there exists a symbolic extension $\pi: Y\to X$ such that $\EA=h^\pi$.
\end{theorem}

In this context it is natural to ask which amenable groups have the comparison property. In fact, up to date there is no characterization of such groups, it not even excluded that all countable amenable groups enjoy this property. We succeeded in proving the comparison property for a large class of groups, the class of all \emph{subexponential groups}.

 \begin{definition}
A countable group $G$ (not necessarily finitely generated) is called \emph{subexponential} if every its finitely generated subgroup $H$ has \emph{subexponential growth}, that is, there exists a generator $R$ of $H$ such that $|(R\cup R^{-1})^n|$ grows subexponentially,\footnote{It is very easy to see that the subexponential growth of a finitely generated group $G$ implies subexponential growth of $|K^n|$ for any finite set $K\subset G$ and thus does not depend on the choice of a finite generator.} i.e.\
$$
\lim_{n\to\infty}\frac1n\log|(R\cup R^{-1})^n|=0.
$$
\end{definition}

It is known that all subexponential groups are amenable, hence, unlike in the case of residually finite groups we do not need to separately assume amenability.
The following result is \cite[Theorem 6.33]{DZ}:

\begin{theorem}\label{main}
Every subexponential group $G$ has the comparison property.
\end{theorem}

\begin{remark}
It is known (but never published, see \cite{Ph} and also \cite{SG}) that finitely generated groups with a symmetric F\o lner \sq\ satisfying the so-called Tempelman's condition (this includes all nilpotent, in particular Abelian, groups) have the comparison property, but beyond this case not much was known before \cite{DZ}.
The class of subexponential groups discussed here covers all virtually nilpotent groups (which have polynomial growth) but also other groups, for instance those with intermediate growth (subexponential but faster than polynomial). The most known example of such groups is the Grigorchuk group \cite{Gr}. We remark that by a recent result from \cite{BGT} by Breuillard, Green and Tao, the above mentioned ``Tempelman groups'' turn out to be virtually nilpotent, so the above mentioned result from \cite{Ph} is covered by Theorem \ref{main}.
\end{remark}

Combining Theorems \ref{coc} and \ref{main}, we conclude:

\begin{corollary}
Let $\mathcal{H}$ be an \ens\ of an action $(X, G)$, where the group $G$ is subexponential and the action $(X, G)$ has finite topological entropy. Then $\EA$ is a (finite) affine \se\ of the structure $\H$ if and only if there exists a symbolic extension $\pi: Y\to X$ such that $\EA=h^\pi$.
\end{corollary}

\section{Questions for further study} \label{ques}

We end this survey by listing some questions related to the subject discussed here, which may be of independent interest and useful for further study of amenable group actions.

\medskip

Absolutely the following main question remains still open for us.

\begin{problem} \label{main-Q}
Can we claim the full version of the Symbolic Extension Entropy Theorem for general amenable groups $G$, in other words, does Theorem \ref{seet} hold for any amenable $G$?
\end{problem}

We have every reason to believe that a result similar to Theorem \ref{sex-vp} holds.

\begin{problem} \label{vp-Q}
Is it true that $\mathbf{h_\text{sex}} (X, G)= \sup \{h_\text{sex} (\mu): \mu\in \mathcal{M} (X, G)\}$?
\end{problem}

We remark that a positive answer to Question \ref{main-Q} implies the positive answer to Question \ref{vp-Q}. But the reverse implication need not hold. It may happen that the equality in Question \ref{vp-Q} holds while the function $h_{sex}$ is not equal to $\EH$ (in fact, it can only be larger).
\medskip

Obviously, not every countable amenable group is residually finite. But, by Theorem \ref{coc}, we can hope to resolve Problem \ref{main-Q} by answering the following question affirmatively:

\begin{problem} \label{coc-Q}
Is it true that any countable amenable group $G$ has the comparison property?
\end{problem}

It seems that adding any reasonable assumption about the group should help. But in fact, the following question still remains open.

\begin{problem}
Is it true that any residually finite amenable group $G$ has the comparison property? What about (right or left) orderable countable amenable groups?
\end{problem}

Recall that a quasi-symbolic system is a topological joining of a subshift with a zero entropy tiling system. Thus if we had an affirmative answer to the following Problem \ref{ti}, then each quasi-symbolic system would admit a principal symbolic extension, and so we could use Theorem \ref{quasi} to solve Problem \ref{main-Q}.

\begin{problem}\label{ti}
Does any countable amenable group admit a zero entropy tiling system which admits a principal symbolic extension?
\end{problem}

\medskip
In the context of Theorem \ref{dt} it also seems natural to ask:

\begin{problem} \label{tiling}
Is it true that the group $G$ has the comparison property, if any \zd\ free $G$-action has a tiling system as a \tl\ factor?
\end{problem}

We remark that the standard irrational rotation of the circle is a zero entropy free $\mathbb{Z}$-action, and that any extension of a free action is still free, thus there exists a zero entropy free $\mathbb{Z}$-action which is a symbolic system (the simplest such system is called the \emph{Sturmian subshift}). But for the case of a general amenable group $G$ the existence of a free zero entropy subshift is unknown (in view of Theorem \ref{dt}, this is exactly what we are missing to reduce Problem \ref{ti} to Problem \ref{coc-Q}):

\begin{problem} \label{free-sym}
Is there a zero entropy free $G$-action in form of a subshift?
\end{problem}

The following result from \cite{DHZ} may be worthy to be recorded here.

\begin{theorem} \label{free-zd}
There always exists a \zd\ zero entropy free $G$-action.
\end{theorem}

As mentioned in \S \ref{stru}, the small boundary property plays an important role in the definition of \ens\ for $\mathbb{Z}$-actions.
Small boundary property has been discussed for $\mathbb{Z}^k$-actions in \cite{GLT} and for amenable group actions in \cite{KS, LT}, yet the following question remains open:

\begin{problem} \label{202002111823-1}
Which $G$-actions with finite topological entropy admit a refining sequence of finite Borel partitions with small boundaries?
\end{problem}

We end this survey by recording again the following problem from \cite{W0}. Recall that a finite set $T\subset G$ is a \emph{monotile} if there exists $C\subset G$ such that $\{T c: c\in C\}$ partitions the group $G$.

\begin{problem} \label{monotile}
Is it true that any $G$ possesses a F\o lner \sq\ consisting of monotiles?
\end{problem}


We remark that several typical classes of amenable groups, including residually finite amenable groups, satisfy the above mentioned property \cite{W0}.

\section*{Acknowledgements}

Tomasz Downarowicz is supported by NCN (National Science Center, Poland) grant\break2018/30/M/ST1/00061 and by the Wroc\l aw University of
Science and Technology grant 049U/0052/19. Guohua Zhang is supported by NSFC (National Natural Science Foundation of China) Grants 11671094, 11722103 and 11731003.


\bibliographystyle{amsplain} 




\printindex

\end{document}